\documentclass[reqno, 11pt]{amsart}

\usepackage[english]{babel}
\usepackage{amsmath}
\usepackage{amssymb}
\usepackage{enumerate}
\usepackage{ifthen}
\usepackage{bbm}
\usepackage{color}
\usepackage{graphicx}
\usepackage{geometry}
\provideboolean{shownotes}
\setboolean{shownotes}{true}
\usepackage{color}
\usepackage{hyperref}
\usepackage{setspace}
\newtheorem{theorem}{Theorem}[section]
\newtheorem{proposition}[theorem]{Proposition}
\newtheorem{lemma}[theorem]{Lemma}
\newtheorem{corollary}[theorem]{Corollary}
\newtheorem{definition}[theorem]{Definition}

\theoremstyle{remark}
\newtheorem{remark}[theorem]{Remark}

\newcommand{\e}{\varepsilon}
\newcommand{\R}{\mathbb{R}}
\newcommand{\T}{\mathbb{T}}
\newcommand{\N}{\mathbb{N}}

\newcommand{\dive}{\mathop{\mathrm {div}}}

\newcommand{\dx}{\partial_x}
\newcommand{\dt}{\partial_t}
\newcommand{\dxx}{\partial^2_{xx}}
\newcommand{\weakto}{\rightharpoonup}
\newcommand{\weaktos}{\stackrel{*}{\rightharpoonup}}

\newcommand{\rrho}{\sqrt{\rho}}

\newcommand{\we}{w_{\e}}
\newcommand{\bede}{\beta_{\sigma}(u_{\e})}
\newcommand{\bedep}{\beta^{'}_{\sigma}(u_{\e})}
\newcommand{\bedes}{\beta^{''}_{\sigma}(u_{\e})}
\newcommand{\bed}{\beta_{\sigma}(u)}
\newcommand{\bedp}{\beta^{'}_{\sigma}(u)}

\newcommand{\re}{\rho_{\e}}
\newcommand{\ue}{u_{\e}}
\newcommand{\me}{\mu_{\e}}
\newcommand{\ke}{k_{\e}}

\numberwithin{equation}{section}

\subjclass{Primary: 35Q35; Secondary: 76N10, 35D30.}
\keywords{Compressible fluid, Navier-Stokes-Korteweg equations, Weak solutions.}

\begin{document}

\title[\empty]{Global weak solutions of the Navier-Stokes-Korteweg Equations in one dimension}

\author[P. Antonelli]{Paolo Antonelli}
\address[Paolo Antonelli]{\newline
Gran Sasso Science Institute\\ viale Francesco Crispi, 7, 67100\\L'Aquila \\Italy}
\email[]{\href{paolo.antonelli@gssi.it}{paolo.antonelli@gssi.it}}
\author[D. Bresch]{Didier Bresch}
\address[Didier Bresch]{\newline
Universit\'e Savoie Mont Blanc\\ UMR5127 CNRS\\ Laboratoire de Math\'ematiques\
\\ 73376 Le Bourget-du-Lac\\ France}
\email[]{\href{didier.bresch@univ-smb.fr}{didier.bresch@univ-smb.fr}}
\author[S. Spirito]{Stefano Spirito}
\address[Stefano Spirito]{\newline
DISIM - Dipartimento di Ingegneria e Scienze dell’Informazione e Matematica\\ Universit\`a degli Studi dell’Aquila\\ Via Vetoio\\ 67100 L’Aquila\\ Italy}
\email[]{\href{stefano.spirito@univaq.it}{stefano.spirito@univaq.it}}

\begin{abstract}
We prove the global existence of weak solutions of the one-dimensional Navier-Stokes-Korteweg (NSK) equations when the viscosity and the capillarity coefficients are power functions of the density, which may be zero on a set with positive measure. The proofs are based on a truncation argument combined with the Energy estimate and BD Entropy. Notably, we do not require any upper bound on the exponent of the power of the viscosity coefficient. In particular, we are able to consider very degenerate viscosity coefficient and to substantially improve previous results. 
\end{abstract}

\maketitle

\section{Introduction}
Let $T>0$, $\gamma>1$, and $\T$ be the one dimensional flat torus. We consider the Navier-Stokes-Korteweg equations in $(0,T)\times\T$:
\begin{align}
&\dt\rho+\dx(\rho\,u)=0,\, \rho\geq 0,\nonumber\\
&\dt(\rho\,u)+\dx(\rho\,u\,u)-\dx(\mu(\rho)\dx\,u)+\dx \rho^{\gamma}=\rho\dx\left(\dx(k(\rho)\dx\rho)-\frac{k'(\rho)}{2}|\dx\rho|^2\right),\label{eq:nsk}\tag{NSK}
\end{align}
with initial conditions 
\begin{equation}\label{eq:id}\tag{ID}
\begin{aligned}
\rho|_{t=0}&=\rho_0\mbox{ on }\{t=0\}\times\T,\\
\rho\,u|_{t=0}&=\rho_0\,u_0\mbox{ on }\{t=0\}\times\T,
\end{aligned}
\end{equation}
and periodic boundary conditions. The system \eqref{eq:nsk} models the dynamic of a compressible fluid subject to viscous and capillarity effects. In particular, $\rho:(0,T)\times\T\mapsto \R^{+}$ denotes the density and $u:(0,T)\times\T\mapsto \R$ denotes the velocity. 
The quantities 
\begin{equation*}
\begin{aligned}
&\mathcal{S}(\dx u):=\mu(\rho)\dx u,\quad &\mathcal{K}(\rho,\nabla\rho):=\rho\dx(k(\rho)\dx\rho)-\frac{1}{2}\left(k(\rho)+\rho\,k'(\rho)\right)|\dx\rho|^2.
\end{aligned}
\end{equation*}
are the viscous stress term and the Korteweg term, respectively. 
Note that
\begin{equation*}
\dx\mathcal{K}(\rho,\dx\rho)=\rho\dx\left(\dx(k(\rho)\dx\rho)-\frac{k'(\rho)}{2}|\dx\rho|^2\right),
\end{equation*}
and this form of $\mathcal{K}$ has been derived in \cite{DS}.
The functions $\mu$, $k$, are the viscosity coefficient and the capillarity coefficient, respectively. In particular, $\mu(\cdot):[0,\infty)\mapsto[0,\infty)$ is smooth, monotone increasing and possibly such that $\mu(0)=0$, and the function $k(\rho):(0,\infty)\mapsto(0,\infty)$ is smooth. The equations \eqref{eq:nsk} arise in several physical applications, such as water waves, quantum hydrodynamic, shallow water equations, and diffuse interface models. 
In most of the cases the coefficients $\mu$, $k$, are polynomial functions of the density and this is also the case we are interested. In particular, in the present note we assume that 
\begin{equation}\label{eq:def1}
\begin{aligned}
&\mu(\rho)=\rho^{\alpha}\quad\alpha>0,&
&k(\rho)=\rho^{\beta}\quad\beta\in\R.
\end{aligned}
\end{equation}
For the sake of the present introduction, it is convenient to continue considering general coefficients.\\

\noindent The system \eqref{eq:nsk} admits a classical formal energy estimate, which is
\begin{equation}\label{eq:energy}
\frac{d}{dt}\mathcal{E}(t)+\int\mu(\rho)|\dx\,u|^2\,dx=0,
\end{equation}
where
\begin{equation*}
\mathcal{E}(t):=\int\rho\frac{|u|^2}{2}+\frac{\rho^{\gamma}}{\gamma-1}+k(\rho)\frac{|\dx\rho|^2}{2}\,dx.
\end{equation*}
In addition to \eqref{eq:energy}, a crucial tool in the analysis of this type of system is the so-called BD Entropy, which was first discovered in \cite{BD} for the multidimensional compressible Navier-Stokes equations with $\mu(\rho)=\rho$, $\lambda(\rho)=0$ and \cite{BD1} with $\lambda(\rho) = 2(\mu'(\rho)\rho-\mu(\rho))$, see also \cite{Sleu} for the one-dimensional case. Let $\phi(\rho)$ be such that $\rho\phi'(\rho)=\mu'(\rho)$, then formally
\begin{equation}\label{eq:bdgeneral}
\frac{d}{dt}\mathcal{F}(t)+\gamma\int\mu'(\rho)\rho^{\gamma-2}|\dx\rho|^2\,dx+\mathcal{J}(\rho)=0,
\end{equation}
where, denoting $w  = u + \alpha \rho^{\alpha-2} \partial_x \rho$, we define
\begin{equation*}
\begin{aligned}
\mathcal{F}(t)&:=\int\rho\frac{|w|^2}{2}+f(\rho)+k(\rho)\frac{|\dx\rho|^2}{2}\,dx \\
\mathcal{J}(\rho)&:=\int\dx\mu(\rho)\dx\left(\dx(k(\rho)\dx\rho)-\frac{k'(\rho)}{2}|\dx\rho|^2\right)\,dx.
\end{aligned}
\end{equation*}
Since $\mu(\cdot)$ is increasing, it always holds that
\begin{equation*}
\int\mu'(\rho)\rho^{\gamma-2}|\dx\rho|^2\geq 0.
\end{equation*}
Thus, in order to control $\mathcal{F}(t)$, a reasonable requirement on $\mu$ and $k$ is that
\begin{equation}\label{eq:bdterm}
\mathcal{J}(\rho):=\int\dx\mu(\rho)\dx\left(\dx(k(\rho)\dx\rho)-\frac{k'(\rho)}{2}|\dx\rho|^2\right)\,dx\geq 0.
\end{equation}
Following \cite{GLF2012} we introduce the following definition.
\begin{definition}[Strong Coercivity Condition]
The functions $\mu$ and $k$ satisfy the {\em Strong Coercivity Condition} if for any positive $\rho\in H^{2}(\T)$ bounded away from zero,
\begin{equation}\label{eq:sc}\tag{SC}
\mathcal{J}(\rho)\geq 0.
\end{equation}
\end{definition}
In particular, in \cite[Proposition 2.3]{GLF2012} the authors characterize the condition \eqref{eq:sc} when $\mu$ and $k$ satisfy \eqref{eq:def1} by proving that the condition \eqref{eq:sc} holds if and only if 
\begin{equation}\label{eq:abrange}
2\alpha-4\leq \beta\leq 2\alpha-1.
\end{equation}
Moreover, in \cite[Theorem 1.1]{GLF2012} the authors prove existence of finite energy weak solutions for certain non-linear coefficients $\mu$ and $k$. In particular, when the $\mu$ and $k$ are as in \eqref{eq:def1}, the global existence of finite energy weak solutions proved in \cite{GLF2012} is for $\alpha$ and $\beta$ such that \eqref{eq:abrange} holds and one of the following two conditions is satisfied:
\begin{align}
&\alpha<\frac{1}{2}\mbox{ or }\beta<-2,\label{eq:cond1}\\
&\alpha>\frac{2}{3},\,\,2\alpha-3<\beta\leq -1.\label{eq:cond2}
\end{align}
Concerning \eqref{eq:cond1}, we note that \eqref{eq:bdgeneral} implies that the density remains bounded away from zero by a constant depending from the initial data. Thus,  if the initial data are smooth enough and the initial density is bounded away from zero a standard higher-order energy estimates argument gives actually the existence of global in time smooth solutions, see Appendix \ref{teo:appex}. On the other hand, if condition \eqref{eq:cond2} holds, no control on the vacuum set seems to be available, and then \eqref{eq:nsk}-\eqref{eq:id} is satisfied in the weak sense. Note that condition \eqref{eq:cond2} implies that $\frac{2}{3}<\alpha\leq 1$. The aim of this note is to substantially improve the range \eqref{eq:cond2}, in particular with respect to the upped bound of $\alpha$. This is the first result in this direction for the Navier-Stokes-Korteweg equations. Precisely, we prove the following result. We refer to Section \ref{sec:weaksolution} for the definition of weak solutions. 
\begin{theorem}\label{teo:main}
Let $\mu$ and $k$ satisfies \eqref{eq:def1} and $p(\rho):=\rho^{\gamma}$ with $2\gamma>\alpha$. Assume that $\rho_0$ and $u_0$ are such that
\begin{equation}\label{eq:idreal}
\begin{aligned}
&\mathcal{E}(0)<\infty,& &\mathcal{F}(0)<\infty.
\end{aligned}
\end{equation}
Let $\alpha$ and $\beta$ such that \eqref{eq:abrange} is satisfied. Then, if
\begin{equation}\label{eq:condnew}
\begin{aligned}
&2\alpha-3\leq\beta<2\alpha-1,\\
&\alpha>\frac{1}{2},\quad \beta>-2,
\end{aligned}
\end{equation}
there exists a least a finite energy global in time weak solutions of \eqref{eq:nsk}-\eqref{eq:id}.
\end{theorem}

\begin{remark}\label{rem:main}\mbox{}

\begin{itemize}
\item[i)] We stress that no upper bound on $\alpha$ is assumed. Therefore, Theorem \ref{teo:main} includes the case when the viscosity coefficient and the capillarity coefficient, as a consequence of \eqref{eq:condnew}, are very degenerate. 
\item[ii)] The range \eqref{eq:condnew} in Theorem \ref{teo:main} corresponds to the {\em Tame capillarity condition} introduced in \cite{GLF2012}, which generalizes to the case of densities not close to a constant the classical condition that capillarity is controlled by the square of the viscosity, see the discussion in \cite{GLF2012}. Note that in the range 
\begin{equation*}
\alpha>\frac{1}{2},\quad\beta>-2,\quad 2\alpha-4<\beta<2\alpha-3,
\end{equation*}
the BD Entropy \eqref{eq:bdgeneral} still holds, but it is not clear how to prove the global existence of weak solutions. 
\item[iii)] The result can be extended to the case when the domain is the entire real line. On the other hand, the case of domains with boundary is not clear. Indeed, as remarked in \cite{GLF2012} the Strong Coercivity condition may not hold in this case. 
\item[iv)] It would be interesting to understand whether is possible to prove the global existence of weak solutions only with finite energy, so without using the BD Entropy. This seems to be difficult in the generality of the assumptions \eqref{eq:def1}. On the other hand, the case $\alpha>0$ and $\beta<-2$ seems more promising since already the energy implies the non-formation of vacuum. 
\end{itemize}
\end{remark}
The proof of Theorem \ref{teo:main} is based on a compactness argument. In particular, the first step is to construct an approximation method with approximating Energy estimate and BD Entropy. Since the BD Entropy is based on a non-linear transformation of the unknowns, building the approximation is not trivial. To this end, a key point in the proof of Theorem \ref{teo:main} is a generalization of the characterization proved in \cite{GLF2012}. It has already been noted in \cite{AS1,BCNV} that, even in the multidimensional case, if
\begin{equation}\label{eq:quantum}
k(\rho)=\frac{(\mu'(\rho))^2}{\rho}
\end{equation}
then the condition \eqref{eq:sc} holds. The relation \eqref{eq:quantum} has been fundamental in the proof of the results in \cite{AS1} and \cite{BVY}, and it has also been used in the one dimensional case, see \cite{BH1,BH2}. In particular, if we consider power laws coefficients as in \eqref{eq:def1}, the relation \eqref{eq:quantum} implies $\beta=2\alpha-3$, and the upper bound on $\alpha$ obtained in \cite{GLF2012} can be improved, see \cite{BH1, BVY}. Inspired by the condition \eqref{eq:quantum} we prove the following generalization of \cite[Proposition 2.3]{GLF2012}. 
\begin{theorem}\label{teo:generalsc}
Let $\mu:[0,\infty)\mapsto(0,\infty)$ be a positive strictly increasing function such that $\mu(0)=0$. Assume that $\mu$ is smooth on $(0,\infty)$. Then if
\begin{equation}\label{eq:general}
k(\rho):=\rho^{\delta}(\mu'(\rho))^2,
\end{equation}
the {\em Strong Coercivity Condition} \eqref{eq:sc} holds if and only if $-2\leq \delta\leq 1$.
\end{theorem}
Theorem \ref{teo:generalsc} is crucial to build approximating solutions $(\re,\ue)$ such that $\re$ is bounded away from zero and $(\re,\ue)$ satisfies approximating versions of \eqref{eq:energy} and \eqref{eq:bdgeneral}. Indeed, one can consider the general Navier-Stokes-Korteweg equations \eqref{eq:nsk} with viscosity coefficient $\rho^{\alpha}$ being perturbed with 
powers of the density with small exponents, and the capillarity coefficient given by \eqref{eq:general}. Then, the control on the strict positivity of the density will be given by \eqref{eq:bdgeneral}, and thus a smooth sequence of approximating solutions of \eqref{eq:nsk} can be constructed. Next, their convergence must be proved. The main tool in the proof of the convergence of the approximating solutions is the truncation method introduced in \cite{LV} for the multi-dimensional Quantum-Navier-Stokes, and already exploited for some other particular choice of the coefficients, as in \cite{AS1,BVY}. Note since the approximation are smooth, the original argument in \cite{LV} is highly simplified. We mention that in the one-dimensional case also different methods may be used, {\em e. g.} kinetic entropies and Young measures, see \cite{GLF2012, LPT1994, CP2010}. On the other hand, the truncation methods used in this paper can be generalized to the multi-dimensional case. \\
We conclude this introduction by giving a brief account of the state of art of the Navier-Stokes-Korteweg equations. The general multidimensional system in the periodic setting is 
\begin{equation}\label{eq:nskgeneral}
\begin{aligned}
&\partial_t\rho+\dive(\rho u)=0\\
&\partial_t(\rho u)+\dive(\rho u\otimes u)+\nabla p=\nu\dive\mathbb S+\kappa\dive\mathbb K,
\end{aligned}
\end{equation}
where $u:(0,T)\times\T^{d}\mapsto\R^{d}$, $\rho:(0,T)\times \T^{d}\mapsto\R^{+}$, and $\nu,\kappa>0$. In this case
\begin{equation}\label{eq:visc}
\mathbb S:=\mu(\rho)\,D u+\lambda(\rho)\dive u\mathbb I,
\end{equation}
is the viscosity stress tensor and
\begin{equation}\label{eq:cap}
\mathbb K:=\left(\rho\dive(k(\rho)\nabla\rho)-\frac12(\rho k'(\rho)-k(\rho))|\nabla\rho|^2\right)\mathbb I
-k(\rho)\nabla\rho\otimes\nabla\rho
\end{equation}
is the capillarity tensor. In the case  $\kappa=0$ \eqref{eq:nskgeneral} reduces to the system of compressible Navier-Stokes equations. When $d=1$, in the model case $\mu(\rho)=\rho^{\alpha}$, several recent results concerning the global regularity of solutions with density bounded away from the zero have been obtained without any restriction of $\alpha$, we refer in particular to \cite{H, MV2, JX, CDS, CDTP}. If $d=2,3$, in general for large initial data only the existence of finite energy solutions is known. In particular, when the viscosity coefficient $\mu(\rho)$ is chosen degenerating on the vacuum region $\{\rho=0\}$ the Lions-Feireisl theory, \cite{L}, \cite{F}, and the recent approach in \cite{BJ} cannot be used since they rely on the Sobolev bound of the velocity field. If $\lambda(\rho)=\rho\mu'(\rho)-\mu(\rho)$ and $\mu(\rho)$ is a powers of the density, finite energy weak solutions are studied in \cite{BDZ, LX, VY1,  LV, BVY, CCH1}. Well-posedness of regular solutions with vacuum are also studies, see \cite{XZ1, XZ2} and also \cite{LPZ} where the shallow water equations are considered.\\

 When the viscosity $\nu=0$, the system \eqref{eq:nskgeneral} is called Euler-Korteweg. In \cite{BGDD} local well-posedness for smooth, small perturbations of the reference solution $\rho=1, u=0$ has been proved, while in \cite{AH} the result was extended to global irrotational solutions in the same framework. Moreover, when $k(\rho)=1/\rho$ the system \eqref{eq:nskgeneral} is called Quantum Hydrodynamic system (QHD) and arises for example in the description of quantum fluids. The global existence of finite energy weak solutions for the QHD system has been proved for $d=2,3$ in \cite{AM, AM2} without restrictions on the regularity or the size of the initial data. see also \cite{AMZ2} for further result in the multi-dimensional case. Moreover the case $d=1$ has been studied in details, by proving also further properties on the solutions, in \cite{AMZ1}. Finally, non-uniqueness results have been proved in \cite{DFM, MSi}.\\

Concerning the case $\nu,\kappa>0$, for the case $d=1$, besides the aforementioned result in \cite{GLF2012} where global existence of finite energy weak solutions are proved for $\alpha$ and $\beta$ satisfying \eqref{eq:cond1} and \eqref{eq:cond2}, results concerning the global regularity of solutions have been obtained for $\kappa(\rho)$ and $\mu(\rho)$ satisfying \eqref{eq:quantum} in \cite{BH1,BH2}. For $d=2,3$, the global existence of weak solutions  has been proved in \cite{AS1, LV} for the Quantum-Navier-Stokes equations with linear viscosity coefficients, namely $\mu(\rho)=\rho$ and $\kappa(\rho)=1/\rho$. See also \cite{AHS} for the case of non trivial far-field conditions. For the constant capillarity case $\kappa(\rho)=c$ and linear viscosity, the global existence of weak solutions has been proved in \cite{AS4}. Finally, for general viscosity and capillarity coefficient, still satisfying \eqref{eq:quantum}, the global existence of weak solutions has been proved in \cite{BVY}. Note that in the case the coefficients are as in \eqref{eq:def1}, the exponent $\alpha\in(2/3,4)$.\\
 
The analysis of singular limits for the equations \eqref{eq:nskgeneral}-\eqref{eq:cap} has also been considered for certain choice of the coefficients in \cite{BGL, CD, CD1, DFM, GLT, DM1, AHM, ACLS, DM, BGL}, and the analysis of the long time behaviour for the isothermal Quantum-Navier-Stokes equations has been performed in \cite{CCH}. Finally, in \cite{LT} the authors, by using a strategy similar to \cite{LX}, study the existence of global in time finite energy weak solutions to the compressible primitive equations with degenerate viscosity.

\subsection*{Notations}
We use standard notation. In particular, the space of periodic smooth functions with value in $\R$ compactly supported in $(0,T)\times\T$ will be denoted $C^{\infty}_c((0,T)\times\T;\R)$. We will denote with $L^{p}(\T)$ the standard Lebesgue spaces and with $\|\cdot\|_{L^p}$ their norm. The Sobolev space of functions with $k$ distributional derivatives in $L^{p}(\T)$ is $W^{k,p}(\T)$ and in the case $p=2$ we will write $H^{k}(\T)$. The spaces $W^{-k,p}(\T)$ and $H^{-k}(\T)$ denote the dual spaces of $W^{k,p'}(\T)$ and $H^{k}(\T)$ where $p'$ is the H\"older conjugate of $p$. Given a Banach space $X$ we use the classical Bochner space for time dependent functions with value in $X$, namely $L^{p}(0,T;X)$, $W^{k,p}(0,T;X)$ and $W^{-k,p}(0,T;X)$ and when $X=L^p(\Omega)$, the norm of the space $L^{q}(0,T;L^{p}(\Omega))$ is denoted by $\|\cdot\|_{L^{q}_{t}L^{p}_{x}}$. Then, the space $C(0,T;X_w)$ is the space of continuous functions with value in the space $X$ endowed with the weak topology.

\subsection*{Organization of the paper}
The paper is organized as folllows. In Section \ref{sec:sc} we prove Theorem \ref{teo:generalsc} and some addiotional results regarding the {\em Strong Coercivity Condition} which will be used in the sequel. In Section \ref{sec:weaksolution} we give the precise definition of finite energy weak solution for the system \eqref{eq:nsk} and we state precisely the main result of the present paper. Next, in Section \ref{sec:app} we introduce the approximating system and sketch the proof of the global regularity. Finally, in Section \ref{sec:proof} we give the proof of the main result.

\subsection*{Acknowledgments}
The first and the third author gratefully acknowledge the partial support by the Gruppo
Na\-zio\-na\-le per l’Analisi Matematica, la Probabilit\`a e le loro
Applicazioni (GNAMPA) of the Istituto Nazionale di Alta Matematica
(INdAM), and by the PRIN 2020 ``Nonlinear evolution PDEs, fluid
dynamics and transport equations: theoretical foundations and
applications''. The first author is also partially funded by the Italian Ministry of University and Research (MUR) through the Excellence Department Project awarded to GSSI, CUP D13C22003740001.
The second
author gratefully acknowledges the partial support by the Agence Nationale
pour la Recherche grant ANR-23-CE40-0014-01 (ANR Bourgeons).
This work also benefited of the support of the ANR under France 2030 
bearing the reference ANR-23-EXMA-004 (Complexflows project).
The third author is also partially supported by the PRIN2022
``Classical equations of compressible fluids mechanics: existence and
properties of non-classical solutions''and the PRIN2022-PNRR ``Some
mathematical approaches to climate change and its impacts.''

\section{The strong coercivity condition}\label{sec:sc}
This section starts with some observations taken from \cite{AB}. More precisely, in this section we recall some basic inequality 
concerning the {\em Strong Coercivity Condition} and we give the proof of Theorem \ref{teo:generalsc}. We start by recalling the following
\begin{lemma}[Bernis inequality]\label{lem:bernis}
Let $\rho$ be any smooth positive function bounded away from zero. Then,
\begin{equation}\label{eq:bernis}
\frac{1}{9}\int\,\frac{|\dx\rho|^4}{\rho^2}\,dx\leq \int\,|\dxx\,\rho|^2\,dx.
\end{equation}
\end{lemma}
The proof of the lemma is elementary and can be found in \cite{Be}. We remark that the constant $1/9$ is optimal in the sense that if the inequality holds with constant $c>1/9$, then $\rho$ is identically constant, see \cite[Theorem 2.1]{GLF2012}.
\begin{remark}\label{rem:theta0}
Clearly, \eqref{eq:bernis} can applied to $\rho^{\theta}$ for any $\theta\not=0$. In this case it reads
\begin{equation*}
\frac{1}{9}\int\frac{|\dx\rho^{\theta}|^4}{\rho^{2\theta}}\,dx\leq \int|\dxx\rho^{\theta}|^2\,dx. 
\end{equation*}
If $\theta=0$ there is no estimates. Indeed, formally we would get 
\begin{equation*}
\frac{\theta^{2}}{9}\int\frac{|\dx\rho^{\theta}|^4}{\theta^{4}\rho^{2\theta}}\,dx\leq \int\frac{|\dxx\rho^{\theta}|^2}{\theta^2}\,dx,
\end{equation*}
which for $\theta=0$ would lead to 
\begin{equation*}
0\leq \int|\dxx\log\rho|^2\,dx. 
\end{equation*}
\end{remark}
For sake of completeness we state \cite[Proposition 2.3]{GLF2012}, and we give a proof. 
\begin{proposition}\label{cor:bd}
Let $\mu$ and $k$ satisfy \eqref{eq:def1}. The  $\mathcal{J}(\rho)\geq 0$ if and only if $\alpha$ and $\beta$ satisfy
\begin{equation}\label{eq:abc1}
2\alpha-4\leq\beta\leq2\alpha-1.
\end{equation}
Moreover, if 
\begin{equation}\label{eq:abc2}
2\alpha-4<\beta<2\alpha-1.
\end{equation}
there exists a constant $c=c(\alpha,\beta)>0$ such that
\begin{equation}\label{eq:dissipation1}
\int|\dxx\rho^{\theta}|^2\,dx+\int|\dx\rho^{\frac{\theta}{2}}|^4\,dx\leq c\mathcal{J}(\rho).
\end{equation}
when $\theta:=\frac{\alpha+\beta+1}{2}\not=0$. If $\alpha+\beta+1=0$ then there exists a constant $c=c(\alpha,\beta)>0$ such that
\begin{equation*}
\int|\dxx\log\rho|^2\,dx+\int|\partial_x\log\rho|^4\,dx\leq  c\mathcal{J}(\rho).
\end{equation*}
\end{proposition}
\begin{proof}
We start by noting that if $\mu$ and $k$ satisfy \eqref{eq:def1}, then 
\begin{equation*}
\mathcal{J}(\rho)=\int\dxx\rho^{\alpha}\left(\dx(\rho^{\beta}\dx\rho)-\frac{\beta\rho^{\beta-1}}{2}|\dx\rho|^2\right)\,dx.
\end{equation*}
We first consider the case $\theta\not=0$. By a direct computation we have that 
\begin{equation}\label{eq:theta2}
\begin{aligned}
\frac{\theta^2}{\alpha}\mathcal{J}(\rho)&=\frac{(\alpha-\beta-1)(1-\alpha)}{(\alpha+\beta+1)^2}\int\frac{|\dx\rho^{\theta}|^4}{\rho^{2\theta}}\,dx+\int|\dxx\rho^{\theta}|^2\,dx\\
&-\frac{\beta}{\alpha+\beta+1}\int\frac{\dxx\rho^{\theta}}{\rho^{\theta}}|\dx\rho^{\theta}|^2\,dx.
\end{aligned}
\end{equation}
By using integration by parts we have that
\begin{equation*}
\int\frac{|\partial_x\rho^{\theta}|^2}{\rho^{\theta}}\partial_{xx}\rho^{\theta}\,dx=\frac{1}{3}\int\frac{|\partial_x\rho^{\theta|^4}}{\rho^{2\theta}}\,dx.
\end{equation*}
Therefore, from \eqref{eq:theta2} we obtain that
\begin{equation*}
\begin{aligned}
\frac{\theta^2}{\alpha}\mathcal{J}(\rho)&=\left(\frac{(\alpha-\beta-1)(1-\alpha)}{(\alpha+\beta+1)^2}-\frac{\beta}{3(\alpha+\beta+1)}\right)\int\frac{|\partial_x\rho^{\theta}|^4}{\rho^{2\theta}}\,dx+\int|\partial_{xx}\rho^{\theta}|^2\,dx\\
&\geq\left(\frac{(\alpha-\beta-1)(1-\alpha)}{(\alpha+\beta+1)^2}-\frac{\beta}{3(\alpha+\beta+1)}+\frac{1}{9}\right)\int\frac{|\partial_x\rho^{\theta}|^4}{\rho^{2\theta}}\,dx,
\end{aligned}
\end{equation*}
Note that \eqref{eq:abc1} is equivalente to
\begin{equation}\label{eq:1d}
\frac{(\alpha-\beta-1)(1-\alpha)}{(\alpha+\beta+1)^2}-\frac{\beta}{3(\alpha+\beta+1)}+\frac{1}{9}\geq0
\end{equation}
and thus \eqref{eq:sc} holds. Finally, since the constant in Lemma \ref{lem:bernis} is optimal, it follows that the condition $2\alpha-4\leq\beta\leq2\alpha-1$ is also necessary. Moreover, if \eqref{eq:abc2} holds, that  
\begin{equation}\label{eq:1d}
\frac{(\alpha-\beta-1)(1-\alpha)}{(\alpha+\beta+1)^2}-\frac{\beta}{3(\alpha+\beta+1)}+\frac{1}{9}>0
\end{equation}
and thus \eqref{eq:dissipation1} holds. Concerning the case $\theta=0$, again a direct computation shows that
\begin{equation}\label{eq:theta3}
\frac{1}{\alpha}\mathcal{J}(\rho)=2\alpha(1-\alpha)\int|\dx\log\rho|^4\,dx+\int|\dx\log\rho|^2\,dx+(\alpha+1)\int|\dx\log\rho|^2\dx\log\rho\,dx. 
\end{equation}
Since
\begin{equation*}
\int\partial_{xx}\log\rho\partial_x\log\rho\partial_x\log\rho\,dx=0,
\end{equation*}
we have that 
\begin{equation*}
\frac{\mathcal{J}(\rho)}{\alpha}=2\alpha(1-\alpha)\int|\partial_x\log\rho|^4+\int|\dxx\log\rho|^2. 
\end{equation*}
Taking into account Remark \ref{rem:theta0} we obtain that \eqref{eq:sc} holds if and only if  $0<\alpha\leq1$, which is exactly the condition \eqref{eq:abc1} with $\beta=-\alpha-1$. 
\end{proof}
In the proof of Theorem \ref{teo:generalsc} we need the following lemma.
\begin{lemma}\label{lem:generalbernis}
Let $\mu:[0,\infty)\mapsto(0,\infty)$ be a positive strictly increasing function such that $\mu(0)=0$. Assume that $\mu$ is smooth on $(0,\infty)$. Then, for any $\delta\not=1$ and any function $\rho$ smooth and bounded away from zero it holds that
\begin{equation*}
\frac{(\delta-1)^2}{9}\int\rho^{\delta-2}\frac{|\dx\mu(\rho)|^4}{\mu'(\rho)}\,dx\leq \int\rho^{\delta}\mu'(\rho)|\dxx\mu(\rho)|^2\,dx.
\end{equation*}
\end{lemma}
\begin{proof}
We first note that 
\begin{equation*}
\begin{aligned}
\int\frac{\rho^{\delta-2}|\dx\mu(\rho)|^4}{\mu'(\rho)}\,dx
&=\int\rho^{\delta-2}\dx\rho\dx\mu(\rho)|\dx\mu(\rho)|^2\,dx\\
&=\frac{1}{(\delta-1)}\int\dx\rho^{\delta-1}\dx\mu(\rho)\dx\mu(\rho)\dx\mu(\rho)\,dx.\\
\end{aligned}
\end{equation*}
By integrating by parts we obtain that 
\begin{equation*}
\begin{aligned}
\int\frac{\rho^{\delta-2}|\dx\mu(\rho)|^4}{\mu'(\rho)}\,dx
&=-\frac{3}{(\delta-1)}\int\rho^{\delta-1}|\dx\mu(\rho)|^2\dxx\mu(\rho)\,dx\\
&\leq \frac{3}{|\delta-1|}\int\,\rho^{\frac{\delta}{2}}[\mu'(\rho)]^{\frac{1}{2}}|\dxx\mu(\rho)|\frac{\rho^{\frac{\delta}{2}-1}|\dx\mu(\rho)|^2}{[\mu'(\rho)]^{\frac{1}{2}}}\,dx\\
&\leq \frac{3}{|\delta-1|}\left(\int\rho^{\delta}\mu'(\rho)|\dxx\mu(\rho)|^2\,dx\right)^{\frac{1}{2}}\left(\int\frac{\rho^{\delta-2}|\dx\mu(\rho)|^4}{\mu'(\rho)}\,dx\right)^{\frac{1}{2}},
\end{aligned}
\end{equation*}
where in the last line we have used Cauchy-Schwartz inequality. Then, we can easily conclude. 
\end{proof}
We are ready to prove the main result of this section.
\begin{proof}[Proof of Theorem \ref{teo:generalsc}]
We recall that 
\begin{equation*}
\mathcal{J}(\rho):=\int\left[\dx(k(\rho)\dx\rho)-\frac{k'(\rho)}{2}|\dx\rho|^2\right]\dxx\mu(\rho)\,dx.
\end{equation*}
By \eqref{eq:general} we have that 
\begin{equation*}
k'(\rho)=\delta\rho^{\delta-1}(\mu'(\rho))^2+2\rho^{\delta}\mu'(\rho)\mu''(\rho).
\end{equation*}
Then, 
\begin{equation*}
\begin{aligned}
&\dx(k(\rho)\dx\rho)-\frac{k'(\rho)}{2}|\dx\rho|^2\\
=&\dx(\rho^{\delta}\mu'(\rho)\mu'(\rho)\dx\rho)-\frac{\delta\rho^{\delta-1}}{2}|\mu'(\rho)|^2|\dx\rho|^2-\rho^{\delta}\mu'(\rho)\mu''(\rho)|\dx\rho|^2\\
=&\rho^{\delta}\mu'(\rho)\dxx\mu(\rho)+\frac{\delta\rho^{\delta-1}}{2}|\dx\mu(\rho)|^2.
\end{aligned}
\end{equation*}
Therefore, 
\begin{equation}\label{eq:sc1}
\mathcal{J}(\rho)=\int\rho^{\delta}\mu'(\rho)|\dxx\mu(\rho)|^2\,dx+\frac{\delta}{2}\int\rho^{\delta-1}|\dx\mu(\rho)|^2\dxx\mu(\rho)\,dx. 
\end{equation}
By integrating by parts we obtain that 
\begin{equation*}
\frac{\delta}{2}\int\rho^{\delta-1}|\dx\mu(\rho)|^2\dxx\mu(\rho)\,dx=
-\frac{\delta(\delta-1)}{6}\int\rho^{\delta-2}\dx\rho\dx[\mu(\rho)]^3\,dx=-\frac{\delta(\delta-1)}{6}\int\rho^{\delta-2}\frac{|\dx\mu(\rho)|^4}{\mu'(\rho)}\,dx. 
\end{equation*}
Thus, \eqref{eq:sc1} becomes 
\begin{equation}\label{eq:sc2}
\mathcal{J}(\rho)=\int\rho^{\delta}\mu'(\rho)|\dxx\mu(\rho)|^2\,dx
-\frac{\delta(\delta-1)}{6}\int\rho^{\delta-2}\frac{|\dx\mu(\rho)|^4}{\mu'(\rho)}\,dx.
\end{equation}
Note that if $\delta=1$, then $\mathcal{J}(\rho)\geq 0$. Thus we can assume $\delta\not=1$ and by using Lemma \ref{lem:generalbernis} in \eqref{eq:sc2} we obtain that if 
\begin{equation}\label{eq:sc3}
\frac{(\delta-1)^2}{9}-\frac{\delta(\delta-1)}{6}\geq 0
\end{equation}
then $\mathcal{J}(\rho)\geq 0$. Then, noting that \eqref{eq:sc3} is equivalent to $-2\leq \delta\leq 1$ we have proved the sufficiency. The necessity follows by Proposition \ref{cor:bd} after choosing that $\mu$ and $k$ as in \eqref{eq:def1}. 
\end{proof}
Finally, we prove the corollary of Theorem \ref{teo:generalsc} which will be useful to derive uniform bounds on the approximating solutions.
\begin{corollary}\label{cor:gbd}
Let $\mu$ and $k$ as in Theorem \ref{teo:generalsc}. Assume that $-2<\delta<1$ and that $\rho|\mu''(\rho)|\leq C\mu'(\rho)$. Then, there exists a constant $c=c(\delta)>0$ such that
\begin{equation*}
\int\rho^{\delta}(\mu'(\rho))^{3}\left(|\dxx\rho|^2+\frac{|\dx\rho|^4}{\rho^2}\right)\,dx\leq c\, \mathcal{J}(\rho)
\end{equation*}
\end{corollary}
\begin{proof}
Note that if $-2<\delta<1$ we can find a constant $c=c(\delta)>0$ such that 
\begin{equation}\label{eq:sc4}
\int\rho^{\delta}\mu'(\rho)|\dxx\mu(\rho)|^2\,dx
+\int\rho^{\delta-2}\frac{|\dx\mu(\rho)|^4}{\mu'(\rho)}\,dx\leq c\,\mathcal{J}(\rho). 
\end{equation}
Expanding the square in the first term we obtain 
\begin{equation*}
\begin{aligned}
&\rho^{\delta}\mu'(\rho)|\dxx\mu(\rho)|^2=\rho^{\delta}\mu'(\rho)|\dx(\mu'(\rho)\dx\rho)|^2\\
=&\rho^{\delta}[\mu'(\rho)]^3|\dxx\rho|^2
+\rho^{\delta}\mu'(\rho)[\mu''(\rho)]^2|\dx\rho|^4\\
+&2\rho^{\delta}\mu'(\rho)\mu'(\rho)\dxx\rho\mu''(\rho)|\dx\rho|^2.
\end{aligned}
\end{equation*} 
Thus, 
\begin{equation*}
\int\rho^{\delta}[\mu'(\rho)]^3|\dxx\rho|^2\,dx\leq c\,\mathcal{J}(\rho)+2\int\rho^{\delta}\mu'(\rho)\mu'(\rho)|\dxx\rho||\mu''(\rho)||\dx\rho|^2\,dx
\end{equation*}
By using that $|\rho\mu''(\rho)|\leq C\mu'(\rho)$ we get that
\begin{equation*}
\begin{aligned}
\int\rho^{\delta}[\mu'(\rho)]^3|\dxx\rho|^2\,dx&\leq c\,\mathcal{J}(\rho)+2\int\rho^{\delta}\mu'(\rho)\mu'(\rho)|\dxx\rho||\mu''(\rho)||\dx\rho|^2\,dx\\
&\leq c\,\mathcal{J}(\rho)+2C\int\rho^{\delta-1}[\mu'(\rho)]^3\dxx\rho|\dx\rho|^2\,dx\\
&=c\,\mathcal{J}(\rho)+2C\int\rho^{\frac{\delta}{2}}[\mu'(\rho)]^{\frac{3}{2}}\dxx\rho\,\,\frac{\rho^{\frac{\delta}{2}-1}|\dx\mu(\rho)|^2}{[\mu'(\rho)]^{\frac{1}{2}}}\,dx.
\end{aligned}
\end{equation*}
Thus, by Young's inequality we have 
\begin{equation*}
\begin{aligned}
\int\rho^{\delta}[\mu'(\rho)]^3|\dxx\rho|^2\,dx&\leq c\left(\mathcal{J}(\rho)+\int\frac{\rho^{\delta-2}|\dx\mu(\rho)|^4}{\mu'(\rho)}\right)\,dx,
\end{aligned}
\end{equation*}
and we can conclude. 
\end{proof}

\section{Finite energy weak solutions}\label{sec:weaksolution}
In this section we give the definition of finite energy weak solution. The system \eqref{eq:nsk} with $\mu(\rho)$ and $k(\rho)$ as in \eqref{eq:def1} reads as follows 
\begin{align}
&\dt\rho+\dx(\rho\,u)=0, \rho\geq 0\nonumber\\
&\dt(\rho\,u)+\dx(\rho\,u\,u)-\dx(\rho^{\alpha}\dx\,u)+\dx \rho^{\gamma}=\rho\dx\left(\dx(\rho^{\beta}\dx\rho)-\frac{\beta\rho^{\beta-1}}{2}|\dx\rho|^2\right).\label{eq:nsklim}
\end{align}
Formally, we have that 
\begin{align}
&\mathcal{E}(t)+\int_0^t\int\rho^{\alpha}|\dx\,u|^2\,dxds=\mathcal{E}(0),\label{eq:efor}\\
&\mathcal{F}(t)+\gamma\alpha\int_0^t\int\rho^{\alpha+\gamma-3}|\dx\rho|^2\,dxds+\int_0^t\mathcal{J}(\rho(s))\,ds=\mathcal{F}(0)\label{eq:bdfor}
\end{align}
where 
\begin{equation*}
\begin{aligned}
&\mathcal{E}(t):=\int\rho\frac{|u|^2}{2}+\frac{\rho^{\gamma}}{\gamma-1}+\rho^{\beta}|\dx\rho|^2\,dx,\quad
&\mathcal{F}(t):=\int\rho\frac{|w|^2}{2}+\frac{\rho^{\gamma}}{\gamma-1}+\rho^{\beta}|\dx\rho|^2\,dx,
\end{aligned}
\end{equation*}
$w:=u+\alpha\rho^{\alpha-2}\dx\rho$, and $\mathcal{J}(\rho)$ is defined in \eqref{eq:bdterm}. If $2\alpha-4<\beta<2\alpha-1$ and $\theta=\frac{\alpha+\beta+1}{2}\not=0$, the case $\theta=0$ can be treated with minor changes, by Proposition \ref{cor:bd} we have that $\mathcal{J}(\rho)>0$ and 
\begin{equation}\label{eq:bd1}
\iint|\dxx\rho^{\theta}|^2\,dx+\iint|\dx\rho^{\frac{\theta}{2}}|^4\,dx\leq C\mathcal{F}(0).
\end{equation}
 Moreover, there exist constants $\bar{\kappa}_1, \bar{\kappa}_2$ depending only on $\beta$ and $\kappa_1, \kappa_2$ depending on $\alpha$ and $\beta$ such that 
\begin{align}
\dx(\mathcal{K}(\rho,\dx\rho))&=\bar{\kappa}_1\dxx(\rho^{\frac{\beta}{2}+1}\dx\rho^{\frac{\beta}{2}+1})+\bar{\kappa}_2\dx(|\dx\rho^{\frac{\beta}{2}+1}|^2),\label{eq:kweak1}\\
\dx(\mathcal{K}(\rho,\dx\rho))&=\kappa_1\dxx(\rho^{\beta+2-\theta}\dx\rho^{\theta})+\kappa_2\dx(\rho^{\beta+2-\theta}|\dx\rho^{\frac{\theta}{2}}|^2),\label{eq:kweak}
\end{align}
Then, the Korteweg term can be equivalently defined in a weak sense by using \eqref{eq:kweak1} or \eqref{eq:kweak}. In the following definition we choose \eqref{eq:kweak}. 
\begin{definition}\label{def:ws}
We say that $(\rho, u, \mathcal{T})$ with $\rho\geq 0$ is a finite energy weak solution of \eqref{eq:nsk}-\eqref{eq:id} if the following conditions holds:
\begin{itemize}
\item[$(1)$] Integrability condition:
\begin{align*}
&\dx\rho^{\frac{\theta}{2}}\in L^{4}(0,T;L^{4}(\T))& &\rrho\,u\in L^{\infty}(0,T;L^{2}(\T))& &\partial_x\rho^{\alpha-\frac{1}{2}}\in L^{\infty}(0,T;L^{2}(\T))\\
&\rho\in L^{\infty}(0,T;L^{\infty}(\T))& &	\mathcal{T}\in L^{2}(0,T;L^{2}(\T))
\end{align*}
\item[$(2)$] Continuity equation:\\
\noindent For any $\psi\in C^{\infty}_{c}((0,T)\times\T)$ it holds that
\begin{equation*}
\iint\rho\dt\psi+\rrho\,\rrho\,u\dx\psi\,dxdt=0
\end{equation*}
\item[$(3)$] Momentum equation:\\
\noindent For any $\psi\in C^{\infty}_{c}((0,T)\times\T)$ it holds that
\begin{equation*}
\iint\rho\,u\dt\psi+\rho\,u\,u\dx\psi-\rho^{\frac{\alpha}{2}}\mathcal{T}\dx\psi+\rho^{\gamma}\dx\psi+\kappa_{1}\rho^{\beta+2-\frac{\theta}{2}}\dx\rho^{\frac{\theta}{2}}\dxx\psi+\kappa_{2}\rho^{\beta+2-\theta}|\dx\rho^{\frac{\theta}{2}}|^2\dx\psi\,dxdt=0
\end{equation*}
\item[$(4)$] Dissipation:\\
\noindent For any $\psi\in C^{\infty}_{c}((0,T)\times\T)$ it holds that
\begin{equation*}
\begin{aligned}
\iint\rho^{\frac{\alpha}{2}}\mathcal{T}\psi\,dxdt&=-\iint\rho^{\alpha-\frac{1}{2}}\rrho\,u\,\partial_x \psi\,dxdt-\frac{2\alpha}{2\alpha-1}\iint\dx\rho^{\alpha-\frac{1}{2}}\rrho\,u\psi\,dxdt.
\end{aligned}
\end{equation*}
\item[$(5)$] Energy inequality:\\
\noindent There exists a constant $C>0$ such that 
\begin{equation*}
\sup_{t}\int\,\rho|u|^2+|\dx\rho^{\frac{\beta}{2}+1}|^2+\rho^{\gamma}\,dx+ \iint \bigl|\mathcal{T}\bigr|^2\,dxdt\leq C
\end{equation*}
\item[$(6)$] BD Entropy\\
\noindent There exists a constant $C>0$ such that 
\begin{equation*}
\begin{aligned}
&\sup_{t}\int\,\rho|w|^2+|\dx\rho^{\frac{\beta}{2}+1}|^2+\rho^{\gamma}\,dxdt+\iint|\partial_x\rho^{\frac{\gamma+\alpha-1}{2}}|^{2}\,dxdt\\
&+\iint|\dxx\rho^{\theta}|^2\,dx+\iint|\dx\rho^{\frac{\theta}{2}}|^4\,dxdt\leq C
\end{aligned}
\end{equation*}
where $w := u + \alpha \rho^{\alpha-2} \partial_x \rho$.
\item[$(7)$] Initial datum:\\
\noindent The initial data are attained in the following sense when $t$ goes to $0$
\begin{equation*}
\begin{aligned}
&\rho(t)\to\rho_0\in L^{2}(\T),\\
&(\rho\,u)(t)\to \rho_0\,u_0\in L^{2}_{w}(\T).
\end{aligned}
\end{equation*}
\end{itemize}
\end{definition}
Some comments are in order
\begin{remark}\mbox{}

\begin{itemize}
\item[i)] Solutions can be also defined with minor change also when $\theta=0$. In the present, we do not consider this case since it only requires small changes in the proofs. 
\item[ii)] The velocity $u$ appearing in Definition \ref{def:ws} is not uniquely defined on the vacuum set $\{\rho=0\}$. Indeed, it is more reasonable to consider the quantity $\rrho\,u$ that it always well-defined, see Lemma \ref{lem:mom} below. This also motivates the definition of finite energy weak solutions. 
\item[iii)] The quantity $\mathcal{T}$ plays the role of $\rho^{\frac{\alpha}{2}}\dx u$, which again cannot be defined because $u$ is not defined on the vacuum region. In particular, item $(4)$ identifies $\rho^{\frac{\alpha}{2}}\dx u$ in terms of $\rho$ and $\rrho\,u$, see also \cite{LLX}.
\item[iv)] Note that while the BD Entropy is not needed to define the Korteweg tensor in a weak sense, see \eqref{eq:kweak1}, it is crucial to properly define the viscous tensor $\mathcal{T}$ because in item $(4)$ the term $\dx\rho^{\alpha-\frac{1}{2}}$ appears. 
\item[v)] Related to the previous comment and Remark \ref{rem:main}-iv), we note that if $\beta\leq -2$, the energy implies that the density is strictly positive, and thus weak solutions can be defined in the energy space. 
\end{itemize}
\end{remark}

\section{The approximating system}\label{sec:app}

We introduce the approximating system. For $\e>0$ we consider the viscosity and the capillarity coefficient 
\begin{equation}\label{eq:viscosity}
\mu_{\e}(\re):=\re^{\alpha}+\e\re^{\frac{1}{4}},
\end{equation}
and
\begin{equation}\label{eq:capillarity}
\ke(\re):=\frac{\,\re^{\delta}}{\alpha^2}(\me'(\re))^{2},
\end{equation}
with $\delta=\beta-2\alpha+2$. When $k$ is defined as in \eqref{eq:capillarity}, we have that \eqref{eq:nsk} reads as 
\begin{align}
\dt\re+\dx(\re\ue)&=0,\nonumber\\
\dt(\re\ue)+\dx(\re\ue\ue)&-\dx(\me(\re)\dx\ue)+\dx\re^{\gamma}=\frac{\re}{\alpha^2}\dx(\re^{\delta}\mu'_{\e}(\re)\dxx\mu_{\e}(\re))\label{eq:nska}\tag{NSKa}\\
&+\frac{\delta}{2\alpha^2}\re\dx(\re^{\delta-1}\dx\mu_{\e}(\re)\dx\mu_{\e}(\re)).\nonumber
\end{align}
The initial conditions on $\{t=0\}\times\T$ are
\begin{equation}\label{eq:ida}\tag{IDa}
\begin{aligned}
\re|_{t=0}&=\rho_0,\\
(\re\ue)|_{t=0}&=\rho_0\,u_0,
\end{aligned}
\end{equation}
and in space we impose periodic boundary conditions. The approximating system \eqref{eq:nska}-\eqref{eq:ida} admits a unique smooth global solution $(\re,\ue)$ for any fixed $\e>0$. Precisely, the following holds:
\begin{theorem}\label{teo:appex}
Let $k>2$ and assume that $\rho_0\in H^{k+1}(\T)$, $\rho_0\geq C>0$, and $u_0\in H^{k}(\T)$. Then, there exists a unique solution $(\re,\ue)$ such that 
\begin{equation*}
\begin{aligned}
&\re\in C([0,T];H^{k+1}(\T))\cap C^{1}((0,T);H^{k}(\T)),\\
&\ue\in C([0,T];H^{k}(\T))\cap C^{1}((0,T);H^{k-1}(\T)).
\end{aligned}
\end{equation*}
Moreover, there exists a constant $C_{\e}>0$ such that $\re\geq C_{\e}$. 
\end{theorem}
We will not give the full proof of Theorem \ref{teo:appex} since it is quite standard. However, we provide a sketch in the Appendix \ref{app}.\\

We prove estimates uniform in $\e$ which will be needed in the proof of the main theorem. We start with the energy estimate.
\begin{proposition}\label{prop:energy}
Let $(\ue,\re)$ be a smooth solution of \eqref{eq:nska}-\eqref{eq:ida} with $\re$ bounded away from zero. Then, for any $t\in(0,T)$ 
\begin{equation}\label{eq:e}
\mathcal{E}_{\e}(t)+\int_0^t\int\mu_{\e}(\re)|\dx\ue|^2\,dxds=\mathcal{E}_{\e}(0)
\end{equation}
where 
\begin{equation*}
\mathcal{E}_{\e}(t):=\int\,\re\frac{|\ue|^2}{2}+\frac{\re^{\gamma}}{\gamma-1}+\frac{\re^{\delta}(\mu'_{\e}(\re))^2|\dx\re|^2}{\alpha^2}\,dx.
\end{equation*}
\end{proposition}
\begin{proof}
By multiplying the first equation of \eqref{eq:nska} by $\ue$, using the first equation and integrating in space we have 
\begin{equation}\label{eq:e1}
\begin{aligned}
&\frac{d}{dt}\int\re\frac{|\ue|^2}{2}\,dx+\int\mu_{\e}(\re)|\dx\ue|^2\,dx
+\int\dx\re^{\gamma}\ue\,dx\\
+&\frac{1}{\alpha^2}\int\dx(\re\ue)[\re^{\delta}\mu'_{\e}(\re)\dxx\mu_{\e}(\re)+\frac{\delta}{2}\re^{\delta-1}\dx\mu_{\e}(\re)\dx\mu_{\e}(\re)]\,dx=0.
\end{aligned}
\end{equation}
Next, by multiplying the first equation by $\frac{\gamma\re^{\gamma-1}}{\gamma-1}$ and integrating in space we get 
\begin{equation}\label{eq:e2}
\frac{d}{dt}\int\frac{\re^{\gamma}}{\gamma-1}\,dx-\int\dx\re^{\gamma}\ue\,dx=0,
\end{equation}
while, by multiplying again the first equation by $-\re^{\delta}\mu'_{\e}(\re)\dxx\mu_{\e}(\re)-\frac{\delta}{2}\re^{\delta-1}\dx\mu_{\e}(\re)\dx\mu_{\e}(\re)$ and integrating in space we obtain 
\begin{equation}\label{eq:e3}
\frac{d}{dt}\int\frac{\re^{\delta}(\mu'_{\e}(\re))^2|\dx\re|^2}{\alpha^2}\,dx
-\frac{1}{\alpha^2}\int\dx(\re\ue)[\re^{\delta}\mu'_{\e}(\re)\dxx\mu_{\e}(\re)+\frac{\delta}{2}\re^{\delta-1}\dx\mu_{\e}(\re)\dx\mu_{\e}(\re)]\,dx=0.
\end{equation}
Summing up \eqref{eq:e1}, \eqref{eq:e2}, and \eqref{eq:e3},after integrating in time we obtain \eqref{eq:e}.
\end{proof}
Next, we show that the BD Entropy holds.
\begin{proposition}\label{prop:bdentropy}
Let $(\ue,\re)$ be a smooth solution of \eqref{eq:nska}-\eqref{eq:ida} with $\re$ bounded away from zero. Define $\we:=\ue+\dx\phi_{\e}(\re)$ with $\re\phi'_{\e}(\re)=\me'(\re)$. Then, for any $t\in(0,T)$ 
\begin{equation}\label{eq:bdprop}
\mathcal{F}_{\e}(t)+\gamma\int_0^t\int\mu'_{\e}(\re)\re^{\gamma-1}|\dx\re|^2\,dxds+\int_0^t\mathcal{J}(\re(s))\,ds=\mathcal{F}_{\e}(0),
\end{equation}
where 
\begin{equation*}
\mathcal{F}_{\e}(t)=\int\re\frac{|\we|^2}{2}+\re|\dx\phi_{\e}(\re)|^2+\frac{\re^{\gamma}}{\gamma-1}+\frac{\re^{\delta}(\mu'_{\e}(\re))^2|\dx\re|^2}{\alpha^2}\,dx.
\end{equation*}
\end{proposition}
\begin{proof}
We first note that by a direct calculation $(\re,\we)$ solves 
\begin{align}
\dt\re&+\dx(\re\we)=\dxx\mu_{\e}(\re),\nonumber\\
\dt(\re\we)&+\dx(\re\we\we)-\dxx(\mu_{\e}(\re)\we)+\dx(\me(\re)\dx\we)+\dx\re^{\gamma}\label{eq:nskaw}\tag{NSKaw}\\
&=\frac{\re}{\alpha^2}\dx(\re^{\delta}\mu'_{\e}(\re)\dxx\mu_{\e}(\re))+\frac{\delta}{2\alpha^2}\re\dx(\re^{\delta-1}\dx\mu_{\e}(\re)\dx\mu_{\e}(\re)).\nonumber
\end{align}
By multiplying the second equation by $\we$, using the first equation and integrating in space we obtain 
\begin{equation}\label{eq:bdp1}
\begin{aligned}
&\frac{d}{dt}\int\re\frac{|\we|^2}{2}\,dx+\int\dx\re^{\gamma}\we\,dx\\
&+\frac{1}{\alpha^2}\int\dx(\re\we)[\re^{\delta}\mu'_{\e}(\re)\dxx\mu_{\e}(\re)+\frac{\delta}{2}\re^{\delta-1}\dx\mu_{\e}(\re)\dx\mu_{\e}(\re)]\,dx=0.
\end{aligned}
\end{equation}
Next, by multiplying the first equation by $\frac{\gamma\re^{\gamma-1}}{\gamma-1}$ we have 
\begin{equation}\label{eq:bdp2}
\frac{d}{dt}\int\frac{\re^{\gamma}}{\gamma-1}\,dx-\int\dx\re^{\gamma}\we\,dx+\gamma\int\mu'_{\e}(\re)\re^{\gamma-1}|\dx\re|^2\,dx=0.
\end{equation}
Finally, by multiplying again the first equation by $-\re^{\delta}\mu'_{\e}(\re)\dxx\mu_{\e}(\re)-\frac{\delta}{2}\re^{\delta-1}\dx\mu_{\e}(\re)\dx\mu_{\e}(\re)$ and integrating in space we obtain 
\begin{equation}\label{eq:bdp3}
\begin{aligned}
&\frac{d}{dt}\int\frac{\re^{\delta}(\mu'_{\e}(\re))^2|\dx\re|^2}{\alpha^2}\,dx+\mathcal{J}(\re(t))\\
&-\frac{1}{\alpha^2}\int\dx(\re\we)[\re^{\delta}\mu'_{\e}(\re)\dxx\mu_{\e}(\re)+\frac{\delta}{2}\re^{\delta-1}\dx\mu_{\e}(\re)\dx\mu_{\e}(\re)]\,dx=0.
\end{aligned}
\end{equation}
Thus, by summing up \eqref{eq:bdp1}, \eqref{eq:bdp2}, and \eqref{eq:bdp3}, after integrating in time we get \eqref{eq:bdprop}.
\end{proof}
The following Proposition is one of the main consequence of the BD Entropy. 
\begin{proposition}\label{prop:vacuum}
Let $(\ue,\re)$ be a smooth solution of \eqref{eq:nska}-\eqref{eq:ida} with $\re$ bounded away from zero. Then, 
there exists a constant $C>0$ independent on $\e$ such that 
\begin{equation}\label{eq:v}
\e\|\re^{-\frac{1}{4}}\|_{L^{\infty}_{t,x}}+\|\re\|_{L^{\infty}_{t,x}}\leq C. 
\end{equation}
\end{proposition}
\begin{proof}
From \eqref{eq:bdprop} we can infer that for some constant $C>0$ independent on $\e$ we have that 
\begin{equation*}
\sup_{t}\int\frac{(\mu'_{\e}(\re))^2}{\re}|\dx\re|^2\,dx\leq C.
\end{equation*}
By using \eqref{eq:viscosity} we deduce that 
\begin{equation*}
\begin{aligned}
&\sup_{t}\int\re^{2\alpha-3}|\dx\re|^2\,dx\leq C,& 
&\sup_{t}\,\e^2\int\re^{-\frac{5}{2}}|\dx\re|^2\,dx\leq C,
\end{aligned}
\end{equation*}
which easily imply that 
\begin{equation}\label{eq:v1}
\begin{aligned}
&\sup_{t}\int|\dx\re^{\alpha-\frac{1}{2}}|^2\,dx\leq C,& 
&\sup_{t}\,\e^2\int|\dx\re^{-\frac{1}{4}}|^2\,dx\leq C.
\end{aligned}
\end{equation}
Let $M$ be the average of $\rho_0$. Since for any $t\in(0,T)$ 
the average of $\re$ is conserved and it is equal to $M$ we can infer that there exists a constant $C=C(T,M,\mathcal{E}(0))>0$ such that for any $t\in(0,T)$ there exists $x_{t}\in\T$ with the property that 
\begin{equation}\label{eq:v2}
\frac{1}{C}\leq \re(t,x_t)\leq C.
\end{equation}
Therefore, for any $t\in(0,T)$ and $x\in\T$ 
\begin{equation}\label{eq:v3}
\begin{aligned}
\re^{\alpha-\frac{1}{2}}(t,x)&\leq \re^{\alpha-\frac{1}{2}}(t,x_t)+\int_{x_t}^x|\dx\re^{\alpha-\frac{1}{2}}(t,y)|\,dy\\
&\leq \re^{\alpha-\frac{1}{2}}(t,x_t)+\left(\int\,|\dx\re^{\alpha-\frac{1}{2}}(t,y)|^2\,dy\right)^{\frac{1}{2}}\leq C
\end{aligned}
\end{equation}
where in the last inequality we used \eqref{eq:v1} and \eqref{eq:v2}. By the very same argument we get that 
\begin{equation*}
\re^{-\frac{1}{4}}(t,x)\leq \re^{-\frac{1}{4}}(t,x_t)+\int_{x_t}^x|\dx\re^{-\frac{1}{4}}(t,y)|\,dy,
\end{equation*}
and thus by \eqref{eq:v2} and \eqref{eq:v1} we obtain 
\begin{equation}\label{eq:v4}
\e|\re^{-\frac{1}{4}}(t,x)|\leq C+\left(\e^2\int|\dx\re^{-\frac{1}{4}}(t,y)|^2\,dy\right)^{\frac{1}{2}}\leq C.
\end{equation}
Then, \eqref{eq:v3} and \eqref{eq:v4} give \eqref{eq:v}.
\end{proof}
\section{Proof of the main result}\label{sec:proof}
In this section we give the proof of Theorem \ref{teo:main}. We recall that we are assuming that 
\begin{equation}\label{eq:mainab}
\alpha>\frac{1}{2},\quad 2\alpha-3\leq \beta<2\alpha-1.
\end{equation}
We only consider the case $\theta=\frac{\alpha+\beta+1}{2}\not=0$. The case $\alpha+\beta+1=0$ can be treated in the same way with minor changes in the proof. Moreover, we recall that we assume that for some $C>0$
\begin{equation*}
\mathcal{E}_{\e}(0)\leq C;\quad\mathcal{F}_{\e}(0)\leq C.
\end{equation*}
\subsection{Uniform bounds}
We collect the uniform bounds with respect to $\e$.  In the following $C>0$ denotes a constant independent on $\e$. We start with the ones that can be deduced from Proposition \ref{prop:energy} and \ref{prop:bdentropy}.
\begin{equation}\label{eq:ub1}
\begin{aligned}
&\sup_t\,\|\re^{\frac{1}{2}}\,\ue\|_{L^{2}_x}\leq C;& &\sup_{t}\,\|\re^{\frac{\delta}{2}}\mu'_{\e}(\re)\dx\re\|_{L^{2}_x}\leq C;\\ 
&\sup_{t}\|\re\|_{L^{\gamma}_x}\leq C;& &\sup_{t}\|\re^{\frac{1}{2}}\dx\phi_{\e}(\re)\|_{L^{2}_x}\leq C;\\
&\|[\mu'_{\e}(\re)]^{\frac{1}{2}}\dx\ue\|_{L^{2}_{t,x}}\leq C;& &\|\re^{\frac{\delta}{2}}[\mu'_{e}(\re)]^{\frac{1}{2}}\dxx\mu_{\e}(\re)\|_{L^{2}_{t,x}}\leq C;\\
&\|\re^{\frac{\delta-2}{4}}[\mu'_{\e}(\re)]^{\frac{3}{4}}\dx\re\|_{L^{4}_{t,x}}\leq C;& &\|[\mu'_{\e}(\re)]^{\frac{1}{2}}\re^{\frac{\gamma-2}{2}}\dx\re\|_{L^{2}_{t,x}}\leq C.
\end{aligned}
\end{equation}
Note that we also used \eqref{eq:sc4}. Next, by using \eqref{eq:viscosity} we easily infer that 
\begin{equation}\label{eq:ub2}
\begin{aligned}
&\sup_{t}\|\dx\re^{\alpha-\frac{1}{2}}\|_{L^{2}_x}\leq C;& &\|\re^{\frac{\alpha}{2}}\dx\ue\|_{L^{2}_{t,x}}\leq C;& &\|\dx\re^{\frac{\gamma+\alpha-1}{2}}\|_{L^{2}_{t,x}}\leq C. 
\end{aligned}
\end{equation}
We also recall the uniform bounds contained in Proposition \ref{prop:vacuum}. We have that 
\begin{equation}\label{eq:ub3}
\begin{aligned}
&\|\re\|_{L^{\infty}_{t,x}}\leq C;& &\|\e\re^{-\frac{1}{4}}\|_{L^{\infty}_{t,x}}\leq C. 
\end{aligned}
\end{equation}
Next, by Corollary \ref{cor:gbd} we can also infer that 
\begin{equation}\label{eq:diss3}
\iint\re^{\delta}(\mu'_\varepsilon(\re))^{3}\left(|\dxx\re|^2+\frac{|\dx\re|^4}{\re^2}\right)\,dxdt\leq C,
\end{equation}
with $\delta=\beta-2\alpha+2$.
Thus, by recalling \eqref{eq:viscosity} we have that
\begin{equation}\label{eq:ub41}
\begin{aligned}
&\|\dxx\,\re^{\theta}\|_{L^{2}_{t,x}}\leq C,& &\|\dx\re^{\frac{\theta}{2}}\|_{L^{4}_{t,x}}\leq C,
\end{aligned}
\end{equation}
with $\theta=\frac{\alpha+\beta+1}{2}$, and 
\begin{equation}\label{eq:ub4}
\e^{3}\iint|\dxx\re^{\frac{4\delta-1}{8}}|^2\,dxdt+\e^{3}\iint\re^{\delta-\frac{9}{4}}\frac{|\dx\re|^{4}}{\re^{2}}\,dxdt\leq C. 
\end{equation}
\subsection{Convergence lemma}
In this subsection we prove some convergence lemma which will be used in the proof of the main theorem. We start with the following 
\begin{lemma}\label{lem:rhoae}
There exists $\rho\geq 0$ such that 
\begin{equation}\label{eq:rhoae}
\re\to\rho\mbox{ a.e. on }(0,T)\times\T. 
\end{equation}
\end{lemma}
\begin{proof}
We first consider the case $\alpha\geq 1$. By using the continuity equation we have that
\begin{equation*}
\dt\re^{\alpha-\frac{1}{2}}+\dx(\re^{\alpha-\frac{1}{2}}\ue)+\left(\alpha-\frac{3}{2}\right)\re^{\alpha-\frac{1}{2}}\dx\ue=0. 
\end{equation*}
Since $\alpha\geq 1$, by using \eqref{eq:ub1} and \eqref{eq:ub2} we have that 
\begin{equation*}
\begin{aligned}
&\{\dt\re^{\alpha-\frac{1}{2}}\}_{\e}\subset L^{\infty}(0,T;H^{-1}(\T)),& &\{\dx\re^{\alpha-\frac{1}{2}}\}_{\e}\subset L^{\infty}(0,T;L^{2}(\T))
\end{aligned}
\end{equation*}
uniformly in $\e$. Then, by using Aubin-Lions lemma we can conclude that there exists $\rho\geq 0$ such that up to a subsequence not relabelled
\begin{equation*}
\re^{\alpha-\frac{1}{2}}\to\rho^{\alpha-\frac{1}{2}}\mbox{ strongly in }C(0,T;L^{2}(\T)). 
\end{equation*}
Passing to a further subsequence if necessary we obtain \eqref{eq:rhoae}. 

\noindent Next, we consider the case $\frac{1}{2}<\alpha<1$. It is enough to note that 
\begin{equation*}
\dx\re=\frac{\re^{\frac{3}{2}-\alpha}\dx\re^{\alpha-\frac{1}{2}}}{\alpha-\frac{1}{2}}. 
\end{equation*}
Thus, 
\begin{equation*}
\begin{aligned}
&\{\dt\re\}_{\e}\subset L^{\infty}(0,T;H^{-1}(\T)),& &\{\dx\re\}_{\e}\subset L^{\infty}(0,T;L^{2}(\T))
\end{aligned}
\end{equation*}
uniformly in $\e$, and we can conclude as in the previous case. 
\end{proof}
The following lemma is an easy consequence of \eqref{eq:rhoae} and the bounds \eqref{eq:ub1}-\eqref{eq:ub4}. 
\begin{lemma}\label{lem:weakrho}
The following convergences hold: 
\begin{align}
&\re\to\rho\mbox{ in }L^{p}(0,T;L^{p}(\T))\mbox{ for any }p\geq 1\label{eq:slpr}\\
&\re\weaktos\rho\mbox{ in }L^{\infty}(0,T;L^{\infty}(\T)),\label{eq:wlir}\\
&\re^{\theta}\weakto\rho^{\theta}\mbox{ in }L^{2}(0,T;H^{2}(\T)),\label{eq:wh2r}\\
&\dx\re^{\frac{\theta}{2}}\weakto\dx\rho^{\frac{\theta}{2}}\mbox{ in }L^{4}(0,T;L^{4}(\T))\label{eq:wl4r}.
\end{align}
\end{lemma}
Next, we study the convergences concerning the momentum. 
\begin{lemma}\label{lem:mom}
There exist $\Lambda$, $m$ and $\mathcal{T}$ such that 
\begin{align}
&\re^{\frac{1}{2}}\ue\weaktos\Lambda\mbox{ in }L^{\infty}(0,T;L^{2}(\T))\label{eq:weak2u}\\
&\re\ue\to\, m\mbox{ in }L^{p}(0,T;L^{p}(\T))\mbox{ for any }p\in[1,2)\label{eq:strong2m}\\
&[\mu_{\e}(\re)]^{\frac{1}{2}}\dx\ue\weakto\mathcal{T}\mbox{ in }L^{2}(0,T;L^{2}(\T))\label{eq:tauconv}
\end{align}
Moreover, $m=\rho^{\frac{1}{2}}\Lambda$ and $m=0$ on $\{\rho=0\}$.
\end{lemma}
\begin{proof}
We note that by \eqref{eq:ub1} and \eqref{eq:ub3} we have that 
\begin{equation*}
\sup_{t}\|\re\ue\|_{L^{2}_x}\leq C, 
\end{equation*}
with $C>0$ independent on $\e$. Thus, there exist $\Lambda$ and $m$ such that 
\begin{equation}\label{eq:mom1}
\begin{aligned}
&\re^{\frac{1}{2}}\ue\weaktos\Lambda\mbox{ in }L^{\infty}(0,T;L^{2}(\T))\\
&\re\ue\weaktos\,m\mbox{ in }L^{\infty}(0,T;L^{2}(\T)). 
\end{aligned}
\end{equation}
In particular, \eqref{eq:ub3}, \eqref{eq:slpr}, and \eqref{eq:mom1} imply that $m=\rho^{\frac{1}{2}}\Lambda$ a.e. on $(0,T)\times\T$ and that $m=0$ on $\{\rho=0\}$. Therefore, it remains only to prove that \eqref{eq:strong2m}. Let $M\in\N$ and $\phi_{M}$ be the function defined as 
\begin{equation}\label{eq:phim}
\phi_{M}(y):=
\begin{cases}
\quad0\quad &y\in\left[0,\frac{1}{2M}\right)\\ 
\quad2My-1\quad &y\in\left[\frac{1}{2M},\frac{1}{M}\right)\\
\quad1\quad &y\geq\frac{1}{M}.
\end{cases}
\end{equation}
By using \eqref{eq:slpr} and \eqref{eq:mom1} we have that 
\begin{equation}\label{eq:convmt}
\re\phi_{M}(\re)\ue\weaktos \phi_{M}(\rho)\,m\mbox{ in }L^{\infty}(0,T;L^{2}(\T)). 
\end{equation}
By \eqref{eq:ub2} and the definition of $\phi_{M}$ we have that for any fixed $M\in\N$ there exists $C_{M}>0$ independent on $\e$ such that 
\begin{equation*}
\|\dx(\re\phi_{M}(\re)\ue)\|_{L^{2}_{t,x}}\leq C_M.
\end{equation*}
Moreover, by using the momentum equation in \eqref{eq:nska}, the bounds \eqref{eq:ub1}-\eqref{eq:ub4}, and the definition of $\phi_{M}$, after a tedious calculation we also have that 
\begin{equation*}
\|\dt(\re\phi_{M}(\re)\ue)\|_{L^{2}_t(H^{-2}_{x})}\leq C_M.
\end{equation*}
Thus, we obtain from the Aubin-Lions lemma that for any fixed $M\in\N$
\begin{equation}\label{eq:strongmm}
\{\re\phi_{M}(\re)\ue\}_{\e}\mbox{ is precompact in }L^{2}(0,T;L^{2}(\T)). 
\end{equation}
In particular, \eqref{eq:convmt}, \eqref{eq:strongmm}, and a diagonal argument imply that, up to a subsequence not relabelled, it holds that 
\begin{equation}\label{eq:convmml2}
\re\phi_{M}(\re)\ue\to\phi_{M}(\rho)m\mbox{ in }L^{2}(0,T;L^{2}(\T)). 
\end{equation}
To conclude, we start by noting that 
\begin{equation*}
\begin{aligned}
\iint|\re\ue-m|\,dxdt&\leq \iint|(1-\phi_{M}(\re))\re\ue|\,dxdt+\iint|\re\phi_{M}(\re)\ue-\phi_{M}(\rho)m|\,dxdt\\
&+\iint|(1-\phi_{M}(\rho))m|\,dxdt\\
&={I}_{\e}+{II}_{\e}+{III}_{\e}
\end{aligned}
\end{equation*}
Since, by H\"older inequality, \eqref{eq:ub1} and \eqref{eq:mom1} we have that 
\begin{equation*}
{I}_{\e}+{III}_{\e}\leq \frac{C}{\sqrt{M}}
\end{equation*}
for some $C>0$ independent on $\e$ and $M$, by using \eqref{eq:convmml2}, we have that 
\begin{equation*}
\re\ue\to m\mbox{ in }L^{1}(0,T;L^{1}(\T)), 
\end{equation*}
and then \eqref{eq:strong2m} follows by interpolation. Finally, \eqref{eq:tauconv} follows by \eqref{eq:ub1} and standard weak compactness argument. 
\end{proof}
The next lemma will be crucial to deal with the convergence of the Korteweg term. 
\begin{lemma}\label{lem:gradients}
The following convergences hold: 
\begin{align}
&\re^{\frac{\delta}{2}}\mu'_{\e}(\re)\dx\re\to\frac{2\alpha}{\theta}\rho^{\frac{\beta+2-\theta}{2}}\dx\rho^{\frac{\theta}{2}}\mbox{ in }L^{p}(0,T;L^{p}(\T)),\quad1\leq p<4\label{eq:g1}\\
&\re^{\frac{\delta}{2}}(\re\mu''_{\e}(\re))\dx\re\to\frac{2\alpha(\alpha-1)}{\theta}\rho^{\frac{\beta+2-\theta}{2}}\dx\rho^{\frac{\theta}{2}}\mbox{ in }L^{p}(0,T;L^{p}(\T)),\quad1\leq p<4\label{eq:g2}\\
&\re^{\delta+1}[\mu'_{\e}(\re)]^2\dx\re\to\frac{2\alpha^2}{\theta}\rho^{\beta+2-\frac{\theta}{2}}\dx\rho^{\frac{\theta}{2}}\mbox{ in }L^{p}(0,T;L^{p}(\T)),\quad1\leq p<4\label{eq:g3}
\end{align}
\end{lemma} 
\begin{proof}
We first prove that for any $s>0$ 
\begin{equation}\label{eq:gstrong}
\re^{s}\dx\re^{\frac{\theta}{2}}\to \rho^{s}\dx\rho^{\frac{\theta}{2}}\mbox{ in }L^{p}(0,T;L^{p}(\T))\,\, 1\leq p<4. 
\end{equation}
We note that by using Lemma \ref{lem:weakrho} we have that 
\begin{equation}\label{eq:weakg}
\re^{s}\dx\re^{\frac{\theta}{2}}\weakto \rho^{s}\dx\rho^{\frac{\theta}{2}}\mbox{ in }L^{2}(0,T;L^{2}(\T)).
\end{equation}
Then, by expanding the $L^{2}$-norm we obtain 
\begin{equation*}
\begin{aligned}
\|\re^{s}\dx\re^{\frac{\theta}{2}}-\rho^{s}\dx\rho^{\frac{\theta}{2}}\|_{2}^2&=\iint\re^{2s}|\dx\re^{\frac{\theta}{2}}|^2\,dxdt-2\iint\re^{s}\dx\re^{\frac{\theta}{2}}\rho^{s}\dx\rho^{\frac{\theta}{2}}\,dxdt+\iint\rho^{2s}|\dx\rho^{\frac{\theta}{2}}|^2\,dxdt\\
&=-\frac{\theta}{8s}\iint\re^{2s}\dxx\re^{\theta}\,dxdt-2\iint\re^{s}\dx\re^{\frac{\theta}{2}}\rho^{s}\dx\rho^{\frac{\theta}{2}}\,dxdt+\iint\rho^{2s}|\dx\rho^{\frac{\theta}{2}}|^2\,dxdt.
\end{aligned}
\end{equation*}
By using Lemma \ref{lem:weakrho} and \eqref{eq:weakg} we get 
\begin{equation*}
\lim_{\e\to 0}\|\re^{s}\dx\re^{\frac{\theta}{2}}-\rho^{s}\dx\rho^{\frac{\theta}{2}}\|_{2}^2=-\frac{\theta}{8s}\iint\rho^{2s}\dxx\rho^{\theta}\,dxdt-\iint\rho^{2s}|\dx\rho^{\frac{\theta}{2}}|^2\,dxdt=0.
\end{equation*}
Since $s>0$, by using \eqref{eq:ub3} and \eqref{eq:ub4} we obtain \eqref{eq:gstrong}. Next, we prove \eqref{eq:g1}. We note that 
\begin{equation*}
\begin{aligned}
\re^{\frac{\delta}{2}}\mu'_{\e}(\re)\dx\re&=\alpha\re^{\frac{\delta}{2}+\alpha-1}\dx\re+\frac{\e}{4}\re^{\frac{\delta}{2}-\frac{3}{4}}\dx\re\\
&=\frac{2\alpha}{\theta}\re^{\frac{\delta}{2}+\alpha-\frac{\theta}{2}}\dx\re^{\frac{\theta}{2}}+\frac{\e}{4}\re^{\frac{2\delta-1}{4}}\frac{\dx\re}{\sqrt{\re}}
\end{aligned}
\end{equation*}
We recall that $\delta=\beta-2\alpha+2$. In particular, \eqref{eq:mainab} is equivalent to $-1\leq \delta<1$. Therefore, $2\delta-1>\delta-\frac{9}{4}$ and then 
\begin{equation*}
\begin{aligned}
\e\iint\re^{\frac{2\delta-1}{4}}\frac{|\dx\re|}{\sqrt{\re}}\,dxdt&\leq \e^{\frac{1}{4}}\left(\e^{3}\iint\re^{2\delta-1}\frac{|\dx\re|^4}{\re^2}\,dxdt\right)^{\frac{1}{4}}\\
&\leq \e^{\frac{1}{4}}\left(\e^{3}\iint\re^{\delta-\frac{9}{4}}\frac{|\dx\re|^4}{\re^2}\,dxdt\right)^{\frac{1}{4}}\leq C\e^{\frac{1}{4}}
\end{aligned}
\end{equation*}
where in the last line we used \eqref{eq:ub4}. Thus, by noting that 
\begin{equation}\label{eq:deltac}
\frac{\delta}{2}+\alpha-\frac{\theta}{2}=\frac{\beta+2-\theta}{2}\geq\frac{\alpha}{4},
\end{equation}
we can infer that as $\e\to 0$
\begin{equation*}
\re^{\frac{\delta}{2}}\mu'_{\e}(\re)\dx\re\to \frac{2\alpha}{\theta}\rho^{\frac{\beta+2-\theta}{2}}\dx\rho^{\frac{\theta}{2}}\mbox{ in }L^{1}(0,T;L^{1}(\T)).
\end{equation*}
Finally, noting that 
\begin{equation}\label{eq:l4}
\iint\re^{2\delta}[\mu'_{\e}(\re)]^4|\dx\re|^4\,dxdt=\iint\re^{\delta-2}[\mu'_{\e}(\re)]^3|\dx\re|^4\re^{\delta+2}\mu'_{\e}(\re)\,dxdt,
\end{equation}
by using that $\delta\geq-1$, \eqref{eq:viscosity}, \eqref{eq:ub1}, and \eqref{eq:ub3}, we have that 
$\{\re^{\frac{\delta}{2}}\mu'_{\e}(\re)\dx\re\}_{\e}\subset L^{4}(0,T;L^{4}(\T))$ uniformly in $\e$, and since $\rho^{\frac{\beta+2-\theta}{2}}\dx\rho^{\frac{\theta}{2}}\in L^{4}(0,T;L^{4}(\T))$ we obtain \eqref{eq:g1} by interpolating. Next, the proof of \eqref{eq:g2} is analogous to the of \eqref{eq:g1} after noting that 
$|\re\mu''_{\e}(\re)|\leq C_{\alpha}\mu'_{\e}(\re)$. Finally, we prove  \eqref{eq:g3}. We have that 
\begin{equation}\label{eq:last1}
\begin{aligned}
\re^{\delta+1}[\mu'_\varepsilon(\re)]^{2}\dx\re&=\frac{2\alpha^2}{\theta}\re^{\beta+2-\frac{\theta}{2}}\dx\re^{\frac{\theta}{2}}
+\frac{\alpha\e}{2}\re^{\delta+\alpha-\frac{3}{4}}\dx\re+\frac{\e^{2}}{16}\re^{\delta-\frac{1}{2}}\dx\re.
\end{aligned}
\end{equation}
Then, by using \eqref{eq:mainab} we have that $\beta+2-\frac{\theta}{2}>0$. Thus, \eqref{eq:gstrong} implies that 
\begin{equation}\label{eq:last}
\frac{2\alpha^2}{\theta}\re^{\beta+2-\frac{\theta}{2}}\dx\re^{\frac{\theta}{2}}\to\frac{2\alpha^2}{\theta}\rho^{\beta+2-\frac{\theta}{2}}\dx\rho^{\frac{\theta}{2}}\mbox{ in }L^{p}(0,T;L^{p}(\T))\,\, 1\leq p<4.
\end{equation}
Next, by H\"older inequality we get 
\begin{equation*}
\begin{aligned}
\e\iint\re^{\delta+\alpha-\frac{3}{4}}|\dx\re|\,dxdt\leq\e^{\frac{1}{4}}\left(\e^3\iint\re^{4\delta+4\alpha-3}|\dx\re|^4\,dxdt\right)^{\frac{1}{4}}
\end{aligned}
\end{equation*}
Then, by using \eqref{eq:mainab} we have that $4\delta+4\alpha-1>\delta-\frac{9}{4}$, and thus by using \eqref{eq:ub4} we can infer that 
\begin{equation}\label{eq:last2}
\e\iint\re^{\delta+\alpha-\frac{3}{4}}|\dx\re|\,dxdt\leq C\,\e^{\frac{1}{4}}.
\end{equation}
Finally, we consider the last term in \eqref{eq:last1}. By using \eqref{eq:ub3} and H\"older inequality we have 
\begin{equation*}
\begin{aligned}
\e^{2}\iint\re^{\delta-\frac{1}{2}}|\dx\re|\,dxdt&\leq \e\iint\re^{\delta-\frac{1}{4}}|\dx\re|\,dxdt\\
&\leq \e^{\frac{1}{4}}\left(\e^{3}\iint\re^{4\delta+1}\frac{|\dx\re|^4}{\re^{2}}\,dxdt\right)^{\frac{1}{4}}.
\end{aligned}
\end{equation*}
In particular, by \eqref{eq:mainab}, we have that $4\delta+1>\delta-\frac{9}{4}$, and thus by \eqref{eq:ub3} and \eqref{eq:ub4} we have 
\begin{equation}\label{eq:last3}
\e^{2}\iint\re^{\delta-\frac{1}{2}}|\dx\re|\,dxdt\leq C\,\e^{\frac{1}{4}}.
\end{equation}
Therefore, from \eqref{eq:last2} and \eqref{eq:last3} we have that 
\begin{equation*}
\left\|\re^{\delta+1}[\mu'_{\e}(\re)]^{2}\dx\re-\frac{2\alpha^2}{\theta}\rho^{\beta+2-\frac{\theta}{2}}\dx\rho^{\frac{\theta}{2}}\right\|_{L^{1}_{t,x}}\leq \frac{2\alpha^2}{\theta}\|\re^{\beta+2-\frac{\theta}{2}}\dx\re^{\frac{\theta}{2}}-\rho^{\beta+2-\frac{\theta}{2}}\dx\rho^{\frac{\theta}{2}}\|_{L^{1}_{t,x}}+C\,\e^{\frac{1}{4}}.
\end{equation*}
By using \eqref{eq:last} we obtain \eqref{eq:g3} for $p=1$. Finally, arguing as in \eqref{eq:l4}, we have that 
$\{\re^{\delta+1}[\mu'_{\e}(\re)]^{2}\dx\re\}_{\e}\subset L^{4}(0,T;L^{4}(\T))$ uniformly in $\e$, and since $\rho^{\beta+2-\frac{\theta}{2}}\dx\rho^{\frac{\theta}{2}}\in L^{4}(0,T;L^{4}(\T))$, by interpolating we obtain \eqref{eq:g3} for $1<p<4$.
\end{proof}
Finally, conclude with a proposition which will be the main tool in the proof of the main theorem.
\begin{proposition}\label{prop:trunc}
Let $f\in C^{1}\cap L^{\infty}(\R)$ and define 
\begin{equation*}
u:=
\begin{cases}
\frac{m}{\rho}\quad&\mbox{ on }\{\rho>0\},\\
0\quad&\mbox{ on }\{\rho=0\},
\end{cases}
\end{equation*}
Then, as $\e\to 0$, 
\begin{align}
&\re^{s}\,f(\ue)\to\rho^s\,f(u)\mbox{ in }L^{p}(0,T;L^{p}(\T)),\quad1\leq p\mbox{ and }s\geq 0,\label{eq:tr1}\\
&\mu_{\e}(\re)\,f(\ue)\to\rho^\alpha\,f(u)\mbox{ in }L^{p}(0,T;L^{p}(\T)),\quad1\leq p\label{eq:tr2}\\
&\re\ue\,f(\ue)\to\rho\,u\,f(u)\mbox{ in }L^{p}(0,T;L^{p}(\T)),\quad1\leq p<2,\label{eq:tr3}\\
&\re^{\frac{\delta}{2}}\mu'_{\e}(\re)\dx\re\,f(\ue)\to\frac{2\alpha}{\theta}\rho^{\frac{\beta+2-\theta}{2}}\dx\rho^{\frac{\theta}{2}}\,f(u)\mbox{ in }L^{p}(0,T;L^{p}(\T)),\quad1\leq p<4,\label{eq:tr4}\\
&\re^{\delta+1}[\mu'_{\e}(\re)]^2\dx\re\,f(\ue)\to\frac{2\alpha^2}{\theta}\rho^{\beta+2-\frac{\theta}{2}}\dx\rho^{\frac{\theta}{2}}\,f(u)\mbox{ in }L^{p}(0,T;L^{p}(\T)),\quad1\leq p<4,\label{eq:tr5}
\end{align}
\end{proposition}
\begin{proof}
For $i=1,...,5$ and $\e>0$ we define the functions 
\begin{equation*}
\begin{aligned}
&g^{\e}_{i}:=\begin{cases}\re^{s}\,&i=1\\ \mu_{\e}(\re)\,&i=2\\ \re\ue\,&i=3\\ \re^{\frac{\delta}{2}}\mu'_{\e}(\re)\dx\re\,&i=4\\ \re^{\delta+1}[\mu'_{\e}(\re)]^2\dx\re\,&i=5\end{cases}&\quad&
g_{i}:=\begin{cases}\rho^{s}\,&i=1\\ \rho^{\alpha}\,&i=2\\ \rho\,u\,&i=3\\ \frac{2\alpha}{\theta}\rho^{\frac{\beta+2-\theta}{2}}\dx\rho^{\frac{\theta}{2}}\,&i=4\\ \frac{2\alpha^2}{\theta}\rho^{\beta+2-\frac{\theta}{2}}\dx\rho^{\frac{\theta}{2}}\,&i=5\end{cases}
\end{aligned}
\end{equation*}
Be using Lemma \ref{lem:rhoae}, Lemma \ref{lem:mom}, and Lemma \ref{lem:gradients}, up to a subsequence if necessary, we have that for any $i=1,...,5$
\begin{equation*}
g_{i}^{\e}\to g_i\mbox{ a.e. on }(0,T)\times\T. 
\end{equation*}
Moreover, it also holds that $g_{i}=0$ on $\{\rho=0\}$ for any $i=1,...,5$. Indeed, this is obvious for $i=1,2$. For $i=3$ it is proved in Lemma \ref{lem:mom} and for $i=4,5$ it follows by \eqref{eq:ub41} and the fact that because of \eqref{eq:mainab}
\begin{equation*}
\begin{aligned}
&\beta+2-\theta>0,& &\beta+2-\frac{\theta}{2}>0.
\end{aligned}
\end{equation*}
Therefore, for any $i=1,...,5$ we have that a.e. on $\{\rho>0\}$ 
\begin{equation*}
g_{\e}^{i}\,f(\ue)\to\,g^{i}\,f(u)\mbox{ as }\e\to0
\end{equation*}
and a.e. on $\{\rho=0\}$ it holds that 
\begin{equation*}
|g^{i}_{\e}\,f(\ue)|\leq \|f\|_{\infty}|g^{i}_{\e}|\to 0\mbox{ as }\e\to0.
\end{equation*}
Then, 
\begin{equation*}
g_{\e}^{i}\,f(\ue)\to g^{i}\,f(u)\mbox{ a.e. on }(0,T)\times\T.
\end{equation*}
Finally, by using \eqref{eq:ub1}-\eqref{eq:ub4} and Vitali's convergence theorem we conclude. 
\end{proof}

\subsection{Proof of Theorem \ref{teo:main}}\mbox{}
\bigskip

We are now ready to prove the main result of the paper.
 
\begin{proof}[Proof of Theorem \ref{teo:main}]\mbox{}
\bigskip

We prove the result only for smooth initial data $\rho_0$ and $u_0$ with $\rho_0\geq c>0$ for some constant $c>0$. The case of initial data satisfying the the assumption \eqref{eq:idreal} follows by a simple approximation argument. Moreover, we prove only items (3) and (4) of Definition \ref{def:ws} since again the rest of the items follow by standard arguments. We start by introducing the truncations. Let $\sigma>0$ and $\{\beta_{\sigma}\}_{\sigma}$ be a sequence of function such that $\beta_{\sigma}:\R\mapsto \R$, $\beta_{\sigma}\in C^{2}(\R)$ and 

\begin{equation}\label{eq:beta}
|\beta_{\sigma}(t)|\leq \frac{1}{\sigma},\quad \beta_\sigma(t) \underset{\sigma \rightarrow 0}{ \longrightarrow t},  \quad \beta^{'}_\sigma(t) \underset{\sigma \rightarrow 0}{ \longrightarrow 1},  \quad| \beta^{'}_\sigma(t) | \le 1, \quad | \beta^{''}_\sigma (t) | \le C\sigma.
\end{equation}
with $C>0$ independent on $\sigma >0$. Next, since the initial data are smooth, by Theorem \ref{teo:appex}, there exists a unique regular solution $(\re,\ue)$ of the approximating system \eqref{eq:nska}-\eqref{eq:ida} with $\re\geq C_\e$. In particular, for any fized $\sigma>0$, the pair $(\re,\bede)$ satisfies
\begin{equation}\label{eq:truncated}
\begin{aligned}
\dt(\re\bede)+&\dx(\re\ue\bede)-\dx(\mu_{\e}(\re)\dx\ue\bedep)+\mu_{\e}(\re)\dx\ue\bedes+(\dx\re^{\gamma})\bedep\\
=&\frac{\re}{\alpha^2}\dx(\re^{\delta}\mu'_{\e}(\re)\dxx\mu_{\e}(\re))\bedep+\frac{\delta}{2\alpha^2}\re\dx(\re^{\delta-1}\dx\mu_{\e}(\re)\dx\mu_{\e}(\re))\bedep.
\end{aligned}
\end{equation}
Let $\psi\in C^{\infty}((0,T)\times\T)$, then by integrating by parts we get 
\begin{equation*}
\begin{aligned}
&\iint\re\dx(\re^{\delta}\mu'_{\e}(\re)\dxx\mu_{\e}(\re))\bedep\psi\,dxdt+\frac{\delta}{2}\iint\re\dx(\re^{\delta-1}\dx\mu_{\e}(\re)\dx\mu_{\e}(\re))\bedep\psi\,dxdt\\
=&\iint\re^{\delta+1}\mu'_{\e}(\re)\dx\mu_{\e}(\re)\bedep\dxx\psi\,dxdt
+\frac{(\delta+3)}{2}\iint\re^{\delta}\dx\mu_{\e}(\re)\dx\mu_{\e}(\re)\bedep\dx\psi\,dxdt\\
+&\iint\re^{\delta}(\re\mu''_{\e}(\re))\dx\re\dx\mu_{\e}(\re)\bedep\dx\psi\,dxdt
+\iint\re^{\delta+1}\mu'_{\e}(\re)\dx\mu_{\e}(\re)\bedes\dx\ue\dx\psi\,dxdt\\
-&\iint\re^{\delta+1}\mu'_{\e}(\re)\dxx\mu_{\e}(\re)\bedes\dx\ue\psi\,dxdt
+\left(\frac{1-\delta}{2}\right)\iint\re^{\delta}\dx\mu_{\e}(\re)\dx\mu_{\e}(\re)\bedes\dx\ue\psi\,dxdt.
\end{aligned}
\end{equation*}
Thus, multiplying \eqref{eq:truncated} by $\psi$, integrating by parts, we obtain
\begin{equation}\label{eq:final}
\begin{aligned}
&\iint\,\re\bede\dt\psi+\re\ue\bede\dx\psi\,dxdt-\mu_{\e}(\re)\dx\ue\bedep\dx\psi+\re^{\gamma}\bedep\dx\psi\,dxdt\\
+\frac{1}{\alpha^2}&\iint\re^{\delta+1}\mu'_{\e}(\re)\dx\mu_{\e}(\re)\bedep\dxx\psi\,dxdt
+\frac{(\delta+3)}{2\alpha^2}\iint\re^{\delta}\dx\mu_{\e}(\re)\dx\mu_{\e}(\re)\bedep\dx\psi\,dxdt\\
+&\frac{1}{\alpha^2}\iint\re^{\delta}(\re\mu''_{\e}(\re))\dx\re\dx\mu_{\e}(\re)\bedep\dx\psi\,dxdt=\frac{1}{\alpha^2}(\langle R_{\sigma,\psi},\psi\rangle+\langle R_{\sigma,\dx\psi},\dx\psi\rangle)
\end{aligned}
\end{equation}
where 
\begin{equation}\label{eq:remainder}
\begin{aligned}
R^{\e}_{\sigma,\psi}=&-\mu_{\e}(\re)\dx\ue\bedes+\re^{\delta+1}\mu'_{\e}(\re)\dxx\mu_{\e}(\re)\bedes\dx\ue\\
&-\left(\frac{1-\delta}{2}\right)\re^{\delta}\dx\mu_{\e}(\re)\dx\mu_{\e}(\re)\bedes\dx\ue-\re^{\gamma}\bedes\dx\ue\\
&=R^{\e,1}_{\sigma,\psi}+R^{\e,2}_{\sigma,\psi}+R^{\e,3}_{\sigma,\psi}+R^{\e,4}_{\sigma,\psi}\\
R^{\e}_{\sigma,\dx\psi}=&-\re^{\delta+1}\mu'_{\e}(\re)\dx\mu_{\e}(\re)\bedes\dx\ue
\end{aligned}
\end{equation}
The first step it study the limit as $\e\to0$ with $\sigma$ fixed. In particular, by using Proposition \ref{prop:trunc}, Lemma \ref{lem:mom}, and Lemma \ref{lem:gradients}, we obtain that as $\e\to0$ the left-hand side of \eqref{eq:final} converge to 
\begin{equation}\label{eq:final2}
\begin{aligned}
&\iint\,\rho\bed\dt\psi+\rho\,u\bed\dx\psi-\rho^{\frac{\alpha}{2}}\mathcal{T}\bedp\dx\psi+\rho^{\gamma}\bedp\dx\psi\,dxdt\\
&+\kappa_1\iint\rho^{\beta+2-\frac{\theta}{2}}\dx\rho^{\frac{\theta}{2}}\bedp\dxx\psi\,dxdt
+\kappa_2\iint\rho^{\beta+2-\theta}|\dx\rho^{\frac{\theta}{2}}|^2\bedp\dx\psi\,dxdt
\end{aligned}
\end{equation}
where $k_1$ and $k_2$ are the constants in \eqref{eq:kweak}. Next, we study the remainder $R_{\sigma,\psi}$. In particular, we claim that there exists a constant $C>0$ independent of $\sigma$ and $\e$ such that 
\begin{equation}\label{eq:remainder}
\begin{aligned}
\|R^{\e}_{\sigma,\psi}\|_{L^{1}_{t,x}}&\leq C\sigma\\
\|R^{\e}_{\sigma,\dx\psi}\|_{L^{1}_{t,x}}&\leq C\sigma.
\end{aligned}
\end{equation}
We start with $R^{\e}_{\sigma,\psi}$. By using \eqref{eq:ub1}, \eqref{eq:ub3}, and \eqref{eq:beta} we have that 
\begin{equation*}
\|R^{\e,1}_{\sigma,\psi}\|_{L^{1}_{t,x}}\leq C\|[\mu(\re)]^{\frac{1}{2}}\|_{L^{2}_t(L^{2}_x)}\|[\mu(\re)]^{\frac{1}{2}}\dx\,u\|_{L^{2}_t(L^{2}_x)}\|\bedes\|_{L^{\infty}_{t,x}}\leq C\sigma
\end{equation*}
By using \eqref{eq:ub1}, \eqref{eq:ub3}, and \eqref{eq:beta} and the fact that $\re\mu'_{\e}(\re)\leq C\mu_{\re}(\re)$, we have that 
\begin{equation}\label{eq:restriction}
\begin{aligned}
\|R^{\e,2}_{\sigma,\psi}\|_{L^{1}_{t,x}}&\leq C\iint\re^{\delta+1}\mu'_{\e}(\re)|\dxx\mu_{\e}(\re)||\bedes||\dx\ue|\,dxdt\\
&\leq \iint\re^{\frac{\delta+1}{2}}\,\re^{\frac{\delta}{2}}[\mu'_{\e}(\re)]^{\frac{1}{2}}|\dxx\mu_{\e}(\re)|[\re\mu'_{\e}(\re)]|^{\frac{1}{2}}|\dx\ue||\bedes|\,dxdt\\
&\leq C\|\re^{\frac{\delta+1}{2}}\|_{L^{\infty}_{t,x}}\|\re^{\frac{\delta}{2}}[\mu'_{\e}(\re)]^{\frac{1}{2}}\dxx\mu_{\e}(\re)\|_{L^{2}_{t,x}}\|[\mu_{\e}(\re)]^{\frac{1}{2}}\dx\ue\|_{L^{2}_{t,x}}\|\bedes\|_{L^{\infty}_{t,x}}\\
&\leq C\sigma
\end{aligned}
\end{equation}
where in the last line we have used the $\delta\geq -1$. Next, we consider $R^{\e,3}_{\sigma,\psi}$. By using \eqref{eq:ub1}, \eqref{eq:ub3}, \eqref{eq:beta}, and he fact that $\re\mu'_{\e}(\re)\leq C\mu_{\re}(\re)$ we have that 
\begin{equation}\label{eq:restriction1}
\begin{aligned}
\|R^{\e,3}_{\sigma,\psi}\|_{L^{1}_{t,x}}&\leq C\iint\re^{\delta}|\dx\mu_{\e}(\re)|^2|\bedes||\dx\ue|\,dxdt\\
&\leq C\iint\re^{\frac{\delta+1}{2}}\,\re^{\frac{\delta-2}{2}}\frac{|\dx\mu_{\e}(\re)|^2}{[\mu'_{\e}(\re)]^{\frac{1}{2}}}\,[\re\mu'_{\e}(\re)]^{\frac{1}{2}}|\dx\ue|\,dxdt\\
&\leq C \|\re^{\frac{\delta+1}{2}}\|_{L^{\infty}_{t,x}}\|\re^{\frac{\delta-2}{4}}[\mu'_{\e}(\re)]^{\frac{3}{4}}\dx\re\|_{L^{4}_{t,x}}\|[\mu_{\e}(\re)]^{\frac{1}{2}}\dx\ue\|_{L^{2}_{t,x}}\|\bedes\|_{L^{\infty}_{t,x}}\\
&\leq C\sigma.
\end{aligned}
\end{equation}
Note that in the last line we used again that $\delta\geq -1$. Finally, we consider $R^{\e,4}_{\sigma,\psi}$. By using that $2\gamma>\alpha$, \eqref{eq:ub2}, \eqref{eq:ub3}, and \eqref{eq:beta} we have that 
\begin{equation*}
\begin{aligned}
\|R^{\e,4}_{\sigma,\psi}\|_{L^{1}_{t,x}}&\leq \iint\re^{\gamma-\frac{\alpha}{2}}|\bedes|\re^{\frac{\alpha}{2}}|\dx\ue|\,dxdt\\
&\leq \|\re^{\gamma-\frac{\alpha}{2}}\|_{L^{\infty}_{t,x}}\|\bedes\|_{L^{\infty}_{t,x}}\|\re^{\frac{\alpha}{2}}\dx\ue\|_{L^{2}_{t,x}}\\
&\leq C\sigma.
\end{aligned}
\end{equation*}
Next, concerning $R^{\e}_{\sigma,\dx\psi}$, by using that $\re\mu'_{\e}(\re)\leq C\mu_{\re}(\re)$ we obtain
\begin{equation*}
\begin{aligned}
|R^{\e}_{\sigma,\dx\psi}|&=|\re^{\delta+1}\mu'_{\e}(\re)\dx\mu_{\e}(\re)\bedes\dx\ue|\\
&\leq C\,\re^{\delta+2\alpha-1}|\dx\re||\bedes||\dx\ue|+C\,\e^{2}\re^{\delta-\frac{1}{2}}|\dx\re||\bedes||\dx\ue|\\
&=:R^{\e,1}_{\sigma,\dx\psi}+R^{\e,2}_{\sigma,\dx\psi}.
\end{aligned}
\end{equation*}
Recalling the definition of $\delta$ and $\theta$ we have that 
\begin{equation*}
\|R^{\e,1}_{\sigma,\dx\psi}\|_{L^{1}_{t,x}}\leq C\iint\re^{\frac{\beta-2\alpha+3}{2}}|\dx\re^{\theta}|\re^{\frac{\alpha}{2}}|\dx\ue||\bedes|\,dxdt.
\end{equation*}
We note that 
By using \eqref{eq:ub41} and Poicar\'e inequality we can infer that for some $C>0$ independent of $\e$, it holds that
\begin{equation*}
\iint|\dx\re^{\theta}|^2\,dxdt\leq C.
\end{equation*}
Thus, by using also \eqref{eq:ub2} and \eqref{eq:beta} we have that 
\begin{equation*}
\begin{aligned}
\|R^{\e,1}_{\sigma,\dx\psi}\|_{L^{1}_{t,x}}&\leq C\|\re^{\frac{\beta-2\alpha+3}{2}}\|_{L^{\infty}_{t,x}}\|\dx\re^{\theta}\|_{L^{2}_{t,x}}
\|\re^{\frac{\alpha}{2}}\dx\ue\|_{L^{2}_{t,x}}\|\bedes\|_{L^{\infty}_{t,x}}\\
&\leq C\,\sigma,
\end{aligned}
\end{equation*}
where in the last line we used \eqref{eq:mainab}. Concerning $R^{\e,2}_{\sigma,\dx\psi}$ 
have that 
\begin{equation*}
\begin{aligned}
\|R^{\e,2}_{\sigma,\dx\psi}\|_{L^{1}_{t,x}}&\leq C\e^{2}\iint\,\re^{\frac{4\delta+4}{8}}|\dx\re^{\frac{4\delta-1}{8}}|\re^{\frac{1}{8}}|\dx\ue||\bedes|\,dxdt.\\
&\leq \|\re^{\frac{4\delta+4}{8}}\|_{L^{\infty}_{t,x}}\|\bedes\|_{L^{\infty}_{t,x}}\left(\e^{3}\iint|\dx\re^{\frac{4\delta-1}{8}}|^2\,dxdt\right)^{\frac{1}{2}}
\left(\e\iint\re^{\frac{1}{4}}|\dx\ue|^2\,dxdt\right)^{\frac{1}{2}}
\end{aligned}
\end{equation*}
By using \eqref{eq:ub4} and Poicar\'e inequality, \eqref{eq:ub1} together with \eqref{eq:viscosity}, and \eqref{eq:beta} with can infer that 
\begin{equation*}
\|R^{\e,2}_{\sigma,\dx\psi}\|_{L^{1}_{t,x}}\leq C\,\sigma.
\end{equation*}
Therefore, \eqref{eq:remainder}, and then by weak compactness we can infer that there exist measures $\mathcal{R}_{1,\sigma}$ and $\mathcal{R}_{2,\sigma}$ such that as $\e\to 0$
\begin{align*}
&R^{\e}_{\sigma,\psi}\weaktos\,\mathcal{R}_{1,\sigma}\\
&R^{\e}_{\sigma,\dx\psi}\weaktos\,\mathcal{R}_{2,\sigma}
\end{align*}
and their total variations satisfy 
\begin{equation}\label{eq:total}
|\mathcal{R}_{1,\sigma}|(\T)+|\mathcal{R}_{2,\sigma}|(\T)\leq C\,\sigma. 
\end{equation}
Then, we can conlude that as $\e\to 0$, for any $\sigma>0$ it holds that 
\begin{equation}\label{eq:final}
\begin{aligned}
&\iint\,\rho\bed\dt\psi+\rho\,u\bed\dx\psi-\rho^{\frac{\alpha}{2}}\mathcal{T}\bedp\dx\psi+\rho^{\gamma}\bedp\dx\psi\\
&+\kappa_1\iint\rho^{\beta+2-\frac{\theta}{2}}\dx\rho^{\frac{\theta}{2}}\bedp\dxx\psi\,dxdt+\kappa_2\iint\rho^{\beta+2-\theta}|\dx\rho^{\frac{\theta}{2}}|^2\bedp\dx\psi\\
&=\langle\mathcal{R}_{1,\sigma},\psi\rangle+\langle\mathcal{R}_{2,\sigma},\dx\psi\rangle.
\end{aligned}
\end{equation}
Next, we send $\sigma\to 0$. Note that by \eqref{eq:total} the right-hand side of \eqref{eq:final} goes to zero. Then, by using the dominated convergence theorem we obtain that 
\begin{equation*}
\iint\rho\,u\dt\psi+\rho\,u\,u\dx\psi-\rho^{\frac{\alpha}{2}}\mathcal{T}\dx\psi+\rho^{\gamma}\dx\psi+\kappa_{1}\rho^{\beta+2-\frac{\theta}{2}}\dx\rho^{\frac{\theta}{2}}\dxx\psi+\kappa_{2}\rho^{\beta+2-\theta}|\dx\rho^{\frac{\theta}{2}}|^2\dx\psi\,dxdt=0,
\end{equation*}
where $\kappa_1$ and $\kappa_2$ are constants in \eqref{eq:kweak}.\\

Next, we prove item (4). Let $\phi_{M}$ be the sequence of smooth functions defined in \eqref{eq:phim}. Note that by a direct calculation it holds that for any $a>1$ 
\begin{equation}\label{eq:phim1}
\rho^{a}|\phi'_{M}(\rho)|\leq \frac{C}{M^{a-1}}. 
\end{equation}
Moreover, for $\sigma>0$, let $\hat{\beta}_{\sigma}:\R\to\R$ be a sequence of functions such that 
\begin{equation}\label{eq:hb1}
\hat{\beta}_{\sigma}\to 1\mbox{ a.e. },
\end{equation}
and for some $C>0$ independent on $\sigma$, it holds that
\begin{equation}\label{eq:hb2}
|y||\hat{\beta}_{\sigma}(y)|\leq \frac{C}{\sigma};\quad|\hat{\beta}_{\sigma}(y)|\leq 1;\quad|\hat{\beta}'_{\sigma}(y)|\leq\,C\sigma.
\end{equation}
Let $\psi\in C^{\infty}_{c}((0,T)\times\T)$ be a test function. Then, since $(\re,\ue)$ is smooth we have that 
\begin{equation}\label{eq:tau1}
\begin{aligned}
\iint\mu_{\e}(\re)\dx\ue\phi_{M}(\re)\hat{\beta}_{\sigma}(\ue)\psi\,dxdt=&-\iint\mu_{\e}(\re)\ue\phi_{M}(\re)\hat{\beta}_{\sigma}(\ue)\dx\psi\,dxdt\\
&-\iint\mu_{\e}'(\re)\dx\re\ue\phi_{M}(\re)\hat{\beta}_{\sigma}(\ue)\psi\,dxdt\\
&+\langle\bar{\mathcal{R}}_{\sigma,M,\psi}^{\e},\psi\rangle,
\end{aligned}
\end{equation}
where 
\begin{equation}\label{eq:tau2}
\begin{aligned}
\bar{\mathcal{R}}_{\sigma,M,\psi}^{\e}&=-\mu_{\e}(\re)\ue\phi'_{M}(\re)\dx\re\hat{\beta}_{\sigma}(\ue)\psi-\mu_{\e}(\re)\ue\phi_{M}(\re)\hat{\beta}'_{\sigma}(\ue)\dx\ue\psi\\
&=\bar{\mathcal{R}}_{\sigma,M,\psi}^{\e,1}+\bar{\mathcal{R}}_{\sigma,M,\psi}^{\e,2}.
\end{aligned}
\end{equation}
Let $M$ and $\sigma$ be fixed, then by using \eqref{eq:tauconv}, \eqref{eq:hb2} and Proposition \ref{prop:trunc} we can infer that 
\begin{equation*}
\begin{aligned}
&\iint\mu_{\e}(\re)\dx\ue\phi_{M}(\re)\hat{\beta}_{\sigma}(\ue)\psi\,dxdt\to \iint\rho^{\frac{\alpha}{2}}\mathcal{T}\phi_{M}(\rho)\hat{\beta}_{\sigma}(u)\psi\,dxdt,\\
&\iint\mu_{\e}(\re)\ue\phi_{M}(\re)\hat{\beta}_{\sigma}(\ue)\dx\psi\,dxdt\to \iint\rho^{\alpha-\frac{1}{2}}\Lambda\phi_{M}(\rho)\hat{\beta}_{\sigma}(u)\dx\psi\,dxdt.
\end{aligned}
\end{equation*}
Moreover, by \eqref{eq:viscosity} we have that 
\begin{equation*}
\begin{aligned}
\iint\mu_{\e}'(\re)\dx\re\ue\phi_{M}(\re)\hat{\beta}_{\sigma}(\ue)\psi\,dxdt&=\frac{2\alpha}{2\alpha-1}\iint\dx\re^{\alpha-\frac{1}{2}}\re^{\frac{1}{2}}\ue\phi_{M}(\re)\hat{\beta}_{\sigma}(\ue)\psi\,dxdt\\
&+\frac{\e}{4}\iint\re^{-\frac{3}{4}}\dx\re\ue\phi_{M}(\re)\hat{\beta}_{\sigma}(\ue)\psi\,dxdt.
\end{aligned}
\end{equation*}
By using \eqref{eq:hb2}, the definition of $\phi_{M}$ and  Proposition \ref{prop:trunc} we have that as $\e\to 0$ it holds that 
\begin{equation*}
\re^{\frac{1}{2}}\ue\phi_{M}(\re)\hat{\beta}_{\sigma}(\ue)\to 
\phi_{M}(\rho)\Lambda\hat{\beta}_{\sigma}(u)\mbox{ in }L^{2}(0,T;L^{2}(\T)).
\end{equation*}
Since, it also holds that 
\begin{equation*}
\dx\re^{\alpha-\frac{1}{2}}\weakto\dx\rho^{\alpha-\frac{1}{2}}\mbox{ in }L^{2}(0,T;L^{2}(\T)),
\end{equation*}
we can conclude that as $\e\to 0$
\begin{equation*}
\frac{2\alpha}{2\alpha-1}\iint\dx\re^{\alpha-\frac{1}{2}}\re^{\frac{1}{2}}\ue\phi_{M}(\re)\hat{\beta}_{\sigma}(\ue)\psi\,dxdt
\to \frac{2\alpha}{2\alpha-1}\iint\dx\rho^{\alpha-\frac{1}{2}}\Lambda\phi_{M}(\rho)\hat{\beta}_{\sigma}(u)\psi\,dxdt. 
\end{equation*}
Next, by using the definition of $\phi_M$ and \eqref{eq:hb2} we have that 
\begin{equation*}
\begin{aligned}
\frac{\e}{4}\iint\re^{-\frac{3}{4}}\dx\re\ue\phi_{M}(\re)\hat{\beta}_{\sigma}(\ue)\psi\,dxdt&\lesssim \e\iint \re^{\frac{3}{4}-\alpha}\dx\re^{\alpha-\frac{1}{2}}\,\ue\phi_{M}(\re)\hat{\beta}_{\sigma}(\ue)\,\psi\,dxdt\\
&\lesssim\e\frac{M^{\alpha-\frac{3}{4}}}{\delta}\iint|\dx\re^{\alpha-\frac{1}{2}}|\,dxdt\lesssim \e.
\end{aligned}
\end{equation*}
In conclusion, as $\e\to 0$, we get that 
\begin{equation*}
\iint\mu_{\e}'(\re)\dx\re\ue\phi_{M}(\re)\hat{\beta}_{\sigma}(\ue)\psi\,dxdt\to \frac{2\alpha}{2\alpha-1}\iint\dx\re^{\alpha-\frac{1}{2}}\re^{\frac{1}{2}}\ue\phi_{M}(\re)\hat{\beta}_{\sigma}(\ue)\psi\,dxdt.
\end{equation*}
Finally, we study the remainders. We start with $\bar{\mathcal{R}}_{\sigma,M,\psi}^{\e,1}$. By using \eqref{eq:phim}, \eqref{eq:hb2}, and \eqref{eq:ub1} we have that 
\begin{equation*}
\begin{aligned}
\|\bar{\mathcal{R}}_{\sigma,M,\psi}^{\e,1}\|_{L^{1}_{t,x}}
&\leq \iint\re^{\frac{1}{2}}|\phi'_M(\re)||\ue||\hat{\beta}_{\sigma}(\ue)|\,dxdt\\
&+\frac{\e}{4}\iint\re^{\frac{3}{4}-\alpha}|\phi'_M(\rho)||\dx\re^{\alpha-\frac{1}{2}}||\hat{\beta}_{\sigma}(\ue)|\,dxdt\\
&\lesssim \frac{1}{\sigma\,\sqrt{M}}+\e\frac{M^{\alpha-\frac{3}{4}}}{\sigma}.
\end{aligned}
\end{equation*}
Concerning $\bar{\mathcal{R}}_{\sigma,M,\psi}^{\e,2}$ we distinguish two cases, if $\alpha<1$, we have
\begin{equation*}
\begin{aligned}
\|\bar{\mathcal{R}}_{\sigma,M,\psi}^{\e,2}\|_{L^{1}_{t,x}}&\lesssim\iint\re^{\frac{\alpha-1}{2}}\re^{\frac{1}{2}}|\ue|\phi_M(\re)\hat{\beta}_{\sigma}'(\ue)\re^{\frac{\alpha}{2}}|\dx\ue||\psi|\,dxdt\\
&+\e\iint\re^{-\frac{\alpha+3}{4}}\re^{\frac{1}{2}}|\ue|\phi_M(\re)\hat{\beta}_{\sigma}'(\ue)\re^{\frac{\alpha}{2}}|\dx\ue||\psi|\,dxdt\\
&\lesssim M^{\frac{1-\alpha}{2}}\sigma+\e\sigma\,M^{\frac{\alpha+3}{4}},
\end{aligned}
\end{equation*}
while if $\alpha\geq 1$ we have 
\begin{equation*}
\|\bar{\mathcal{R}}_{\sigma,M,\psi}^{\e,2}|\_{L^{1}_{t,x}}\lesssim \sigma+\e\sigma\,M^{\frac{\alpha+3}{4}}.
\end{equation*}
Clearly, the worst case is $\alpha<1$ and thus we only consider this case. By sending $\e\to0$ we have that there exists a measure $\bar{\mathcal{R}}_{\sigma,M,\psi}$ such that 
\begin{equation*}
\bar{\mathcal{R}}_{\sigma,M,\psi}^{\e}\weaktos 
\bar{\mathcal{R}}_{\sigma,M,\psi}
\end{equation*}
and its total variation satisfies 
\begin{equation}\label{eq:varia1}
|\bar{\mathcal{R}}_{\sigma,M,\psi}|(\T)\lesssim 
\frac{1}{\sigma\sqrt{M}}+M^{\frac{1-\alpha}{2}}\sigma
\end{equation}
Thus, as $\e\to 0$, we obtain that 
\begin{equation*}
\begin{aligned}
\iint\rho^{\frac{\alpha}{2}}\mathcal{T}\phi_M(\rho)\hat{\beta}_{\sigma}(u)\psi\,dxdt&=-\iint\rho^{\alpha-\frac{1}{2}}\Lambda\phi_{M}(\rho)\hat{\beta}_{\sigma}(u)\,\psi\,dxdt\\
&-\frac{2\alpha}{2\alpha-1}\iint\dx\rho^{\alpha-\frac{1}{2}}\Lambda\phi_{M}(\rho)\hat{\beta}_{\sigma}(u)\psi\,dxdt\\
&+\langle\bar{\mathcal{R}}_{\sigma,M,\psi},\psi\rangle. 
\end{aligned}
\end{equation*}
Choosing $\sigma_M=\frac{1}{M^{\frac{1}{4}}}$, by using \eqref{eq:varia1}, we get that as $M\to\infty$
\begin{equation*}
\langle\bar{\mathcal{R}}_{\sigma_M,M,\psi},\psi\rangle\to 0,
\end{equation*}
while by using the dominated convergence theorem we obtain (4). 
\end{proof}
\appendix
\section{Proof of Theorem \ref{teo:appex}}\label{app}
To avoid heavy notations we omit the subscript $\e$. Let $\mu(\rho)$ and $k(\rho)$ given by \eqref{eq:viscosity} and \eqref{eq:capillarity} and consider the system 
\begin{align}
\dt\rho+\dx(\rho\,u)&=0,\nonumber\\
\dt(\rho\,u)+\dx(\rho\,u\,u)&-\dx(\mu(\rho)\dx u)+\dx\rho^{\gamma}=\rho\,\dx\!\left(\dx(k(\rho)\dx\rho))\!-\!\frac{k'(\rho)}{2}|\dx\rho|^2\right).\label{eq:nskapp}
\end{align}
Given $(\rho_0,u_0)\in H^{k+1}(\T)\times H^{k}(\T)$, with $k>2$ and $\rho_0$ strictly positive, by using a standard contraction argument, there exists $T^{*}>0$ and a unique local solution 
\begin{equation*}
(\rho,u)\in C([0,T^{*});H^{k+1}(\T))\times C([0,T^{*});H^{k}(\T)) 
\end{equation*}
of \eqref{eq:nskapp}. We note that \eqref{eq:nskapp} can equivalently formulated as 
\begin{align}
&\dt A+\dx(\tilde{\mu}(\rho)\dx u)=-u\dx\,A-\dx\,uA\nonumber\\
\dt u-&\dx(\bar{\mu}(\rho)\dx u)-\dx(\tilde{\mu}(\rho)\dx A)=f_1(\rho)A\label{eq:sysa}\\
&+f_{2}(\rho)A\dx u+\dx(f_{3}(\rho)|A|^{2})-u\dx\,u\nonumber
\end{align}
where 
\begin{equation*}
A(\rho)=\sqrt{\frac{k(\rho)}{\rho}}\,\dx\rho,\quad \tilde{\mu}(\rho)=\sqrt{\rho\,k(\rho)},\quad \bar{\mu}(\rho)=\frac{\mu(\rho)}{\rho},
\end{equation*}
and the $f_{i}$ are suitable smooth functions depending on $\rho$, which can be explicitly computed. Performing an $H^{k}$ energy estimate we can infer that there exists $\Psi: \R^{3}\mapsto \R^{+}$, smooth and monotone increasing in every argument
such that 
\begin{equation}\label{eq:hs}
\begin{aligned}
&\frac{d}{dt}\left(\|A(t)\|_{H^k_x}+\|u(t)\|_{H^k_x}\right)\\
&\leq \left(\Psi\left(\big\|\frac{1}{\rho(t)}\big\|_{\infty}, \|\dx A(t)\|_{L_x^\infty}, \|\dx u(t)\|_{L_x^\infty}\right)+\|\dx^{2}u(t)\|_{L_x^\infty}\right)\left(\|A(t)\|_{H^k_x}+\|u(t)\|_{H^k_x}\right).
\end{aligned}
\end{equation}
We stress that the skew-symmetric structure in \eqref{eq:sysa} is crucial to deduce the previous higher-order energy estimate. Moreover, \eqref{eq:hs} also provides the following blow-up alternative: 
\begin{equation}\label{eq:chartime}
T^{*}<\infty\Leftrightarrow \limsup_{T\to T^{*}}\int_{0}^{T}\left(\Psi\left(\big\|\frac{1}{\rho(t)}\big\|_{L_x^\infty}, \|\dx A(t)\|_{L_x^\infty}, \|\dx u(t)\|_{L_x^\infty}\right)+\|\dx^{2}u(t)\|_{L_x^\infty}\right)\,dt=\infty. 
\end{equation}
Next, we claim that by a bootstrap argument the local solution $(\rho,u)$ can be extended globally in time. In order to do so, we start by noticing that Proposition \ref{prop:energy}, Proposition \ref{prop:bdentropy}, and Proposition \ref{prop:vacuum} yield the following global bounds: 
\begin{equation}\label{eq:e1}
\begin{aligned}
&\rho,\,\rho^{-1}\in L^{\infty}([0,T]\times\T),\\
&u\in L^{\infty}(0,T;L^{2}(\T))\cap L^{2}(0,T;H^{1}(\T)),\\
&A\in L^{\infty}(0,T;L^{2}(\T))\cap L^{2}(0,T;H^{1}(\T)).
\end{aligned}
\end{equation}
The equations for $(\dx\,u,\dx\,A)$ are 
\begin{align}
&\dt\dx A+\dx^2(\tilde{\mu}(\rho)\dx\,u)=-\dx(u\dx A)-\dx(\dx u\,A),
\nonumber\\
\dt\,\dx u&-\dx^{2}(\bar{\mu}(\rho)\dx u)-\dx^2(\tilde{\mu}(\rho)\dx A)=\dx(f_{1}(\rho)\,A)\label{eq:sysa1}\\
&+\dx(f_2(\rho)A\dx u)+\dx^2(f_{3}(\rho)|A|^{2})-\dx(u\dx u).\nonumber
\end{align}
Then, by a classical energy estimate, we have that 
\begin{equation*}
\frac{d}{dt}(\|\dx A\|_{L^2_x}^2+\|\dx u\|_{L^2_x}^2)+\|\dx^2 u\|_{L^2_x}^{2}\lesssim (1+\|\dx A\|_{L^2_x}^2+\|\dx u\|_{L^2_x}^{2}+\|A\|_{L^2_x}\|\dx A\|_{L^2_x}^2)(\|\dx A\|_{L^2_x}^2+\|\dx u\|_{L^2_x}^2),
\end{equation*}
and by using \eqref{eq:e1} and Gr\"onwall's Lemma we obtain that 
\begin{equation}\label{eq:e2}
\begin{aligned}
&A\in L^{\infty}(0,T;H^{1}(\T)),\\
&u\in L^{\infty}(0,T;H^{1}(\T))\cap L^{2}(0,T;H^{2}(\T)).
\end{aligned}
\end{equation}
The next step is to deduce that $A\in L^{2}(0,T;H^{2}(\T))$. We first note that by the first equation in \eqref{eq:sysa} we easily deduce that 
\begin{equation*}
\|\dt A\|_{L^2_x}^{2}\lesssim \|\dx u\|^2_{L^2_x}\|\dx A\|^2_{L^2_x}
+\|\dx^2\,u\|_{L^2_x}^2+\|\dx u\|_{L^2_x}^2. 
\end{equation*}
Therefore, 
\begin{equation}\label{eq:e3}
\dt A\in L^{2}(0,T;L^{2}(\T)). 
\end{equation}
Then, by using the second equation in \eqref{eq:sysa} we obtain 
\begin{equation}\label{eq:s4}
\int_0^T\|\dx^2 A\|_{L^2_x}^2\,dt\lesssim (1+T)\left(\sup_{t}(\|\dx u\|_{L^2_x}^2
+\|\dx A\|_{L^2_x}^2)\right)+\int_0^T\|\dt A\|_{L^2_x}^{2}\,dt
+\int_0^T\|\dx^2 u\|_{L^2_x}^{2}\,dt.
\end{equation}
Thus, using \eqref{eq:e1}, \eqref{eq:e2}, and \eqref{eq:e3} we obtain
\begin{equation}\label{eq:e4}
A\in L^{2}(0,T;H^{2}(\T)). 
\end{equation}
The last step of the bootstrap argument consists in proving the uniform in time $H^{2}$ regularity. The equations for $(\dx^2 u,\dx^2 A)$ are 
\begin{align}
&\dt\dx^2 A+\dx^3(\tilde{\mu}(\rho)\dx\,u)=-\dx^2(u\dx A)-\dx^2(\dx u\,A),
\nonumber\\
\dt\,\dx^2 u&-\dx^{3}(\bar{\mu}(\rho)\dx u)-\dx^3(\tilde{\mu}(\rho)\dx A)=\dx^2(f_{1}(\rho)\,A)\label{eq:sysa2}\\
&+\dx^2(f_2(\rho)A\dx u)+\dx^3(f_{3}(\rho)|A|^{2})-\dx^2(u\dx u).\nonumber
\end{align}
Then, by performing an $L^2$ energy estimate for the system \eqref{eq:sysa2} we have 
\begin{equation*}
\frac{d}{dt}(\|\dx^2 u\|_{L^2_x}^2+\|\dx^2 A\|_{L^2_x}^2)+\|\dx^3 u\|_{L^2_x}^2\lesssim (1+\|\dx u\|_{L^2_x}^2+\|\dx A\|_{L^2_x}^2+\|\dx^2 u\|_{L^2_x}^2+\|\dx^2 A\|_{L^2_x}^2)(\|\dx^2 u\|_{L^2_x}^2+\|\dx^2 A\|_{L^2_x}^2). 
\end{equation*}
By using \eqref{eq:e1}-\eqref{eq:e4} and Gr\"onwall's Lemma we conclude that 
\begin{equation}\label{eq:5}
\begin{aligned}
&A\in L^{\infty}(0,T;H^{2}(\T),\\
&u\in L^{\infty}(0,T;H^{2}(\T))\cap L^{2}(0,T;H^{3}(\T)). 
\end{aligned}
\end{equation}
Finally, by using \eqref{eq:5} and Sobolev embedding we have that 
\begin{equation*}
\dx u,\dx A\in L^{\infty}([0,T]\times\T),\quad \dx^2 u\in L^{2}(0,T;L^{\infty}(\T))
\end{equation*}
and therefore by the blow-up alternative \eqref{eq:chartime} we conclude.


\begin{thebibliography}{10}

\bibitem{AB} T. Alazard and D. Bresch. Functional inequalities and strong Lyapunov functionals for free surface flows in fluid dynamics.
See arXiv:2004.03440

\bibitem{ACLS} P. Antonelli, G. Cianfarani Carnevale, C. Lattanzio, and S. Spirito, {\em Relaxation limit from the quantum Navier-Stokes equations to the quantum drift-diffusion equation}, J. Nonlinear Sci. {\bf 31} (2021), no. 5, pp. Paper No. 71, 32.

\bibitem{AHM} P. Antonelli, L. E. Hientzsch, and P. Marcati, {\em On the low Mach number limit for quantum Navier-Stokes equations}, SIAM J. Math. Anal., {\bf 52}, (2020), no.6, 6105--6139.

\bibitem{AHS} P. Antonelli, L.E. Hientzsch, and S. Spirito, {\em Global existence of finite energy weak solutions to the quantum Navier-Stokes equations with non-trivial far-field behavior}, J Differential Equations {\bf 290} (2021), 147--177.

\bibitem{AM} P. Antonelli and P. Marcati, {\em On the finite energy weak solutions to a system in quantum fluid dynamics,} Comm. Math. Phys., {\bf287} (2009), 657--686.

\bibitem{AM2} P. Antonelli and P. Marcati, \emph{The Quantum Hydrodynamics system in two space dimensions}, Arch. Ration. Mech. Anal., \textbf{203} (2012), 499--527.

\bibitem{AMZ1} P. Antonelli, P. Marcati, and H. Zheng, {\em Genuine Hydrodynamic Analysis to the 1-D QHD system: Existence, Dispersion and Stability}, Commun. Math. Phys., {\bf 383} (2021), pp. 2113--2161.

\bibitem{AMZ2} P. Antonelli, P. Marcati, and H. Zheng, {\em An intrinsically hydrodynamic approach to multidimensional QHD systems}, Arch. Ration. Mech. Anal. {\bf 247} (2) (2023), 1--58.

\bibitem{AS} P. Antonelli and S. Spirito, {\em On the compactness of finite energy weak solutions to the quantum Navier-Stokes equations}, J. Hyperbolic Differ. Equ., \textbf{15} (2018), 133--147.

\bibitem{AS1} P. Antonelli and S. Spirito, \emph{Global existence of finite energy weak solutions of quantum Navier-Stokes equations}, Arch. Ration. Mech. Anal., {\bf3} (2017), 1161--1199.

\bibitem{AS3} P. Antonelli and S. Spirito, \emph{On the compactness of weak solutions to the Navier-Stokes-Korteweg equations for capillary fluids}, Nonlin. Anal. {\bf 187} (2019), 110--124.

\bibitem{AS4}P. Antonelli and S. Spirito, \emph{Global existence of weak solutions to the Navier-Stokes-Korteweg equations}, Ann. Inst. Henri Poincare (C) Anal. Non Lineaire, {\bf39} (2022), 171--200.

\bibitem{AH} C. Audiard and B. Haspot, {\em Global well-posedness of the Euler-Korteweg system for small irrotational data}, Comm. Math. Phys., \textbf{351} (2017), 201--247.

\bibitem{BG} S. Benzoni-Gavage, \emph{Stability of multi-dimensional phase transitions in a Van der Waals fluid}, Nonlin. Anal. Theory Meth. Appl. {\bf 31}, no. 1--2 (1998), 243--263.

\bibitem{BGDD} S. Benzoni-Gavage, R. Danchin and S. Descombes, \emph{On the well-posedness for the Euler-Korteweg model in several space dimensions}, Indiana Univ. Math. J., {\bf 56} (2007), 1499--1579.

\bibitem{Be} F. Bernis, {\em Integral inequalities with applications to nonlinear degenerate parabolic equations}. In: Angell, T.S., Cook, L.P., Kleinman, R.E. Olmstead, W.E. {\em Nonlinear Boundary Value Problems}. SIAM, Philadelphia, 1996

\bibitem{BCNV} D. Bresch, F. Couderc, P. Noble and J.P. Vila, \emph{A generalization of the quantum Bohm identity: Hyperbolic CFL condition for the Euler-Korteweg equations, G\'en\'eralisation de l'identit\'e de Bohm quantique : condition CFL hyperbolique pour \'equations d'Euler–Korteweg.}, Comptes Rendus Math. {\bf 354}, no. 1 (2016), 39--43.

\bibitem{BD} D. Bresch and D. Desjardins, {\em Existence of Global Weak Solutions for a 2D Viscous Shallow Water Equations and Convergence to the Quasi-Geostrophic Model}, Comm. Math. Phys., {\bf238} 2003, 211--223.

\bibitem{BD1} D. Bresch and B. Desjardins, {\em Quelques mod\`eles diffusifs capillaires de type Korteweg}, 
C. R. Acad. Sci. Paris, section m\'ecanique, {\bf 332}, no. 11, 881--886, (2004).

\bibitem{BDL} D. Bresch, B. Desjardins and Chi-Kun Lin, {\em On some compressible fluid models: Korteweg, lubrication, and shallow water systems}, Comm. Part. Differ. Equat., {\bf28} (2003), 843--868.

\bibitem{BDZ} D. Bresch, B. Desjardins and E. Zatorska, {\em  Two-velocity hydrodynamics in Fluid Mechanics, Part II. Existence of global $\kappa$-entropy solutions to compressible Navier-Stokes system with degenerate viscosities}. J. Math. Pures Appl. Volume {\bf 104}, Issue 4, 801--836 (2015).

\bibitem{BJ} D. Bresch and P.-E. Jabin, \emph{Global existence of weak solutions for compressible Navier-Stokes equations; thermodinamically unstable pressure and anisotropic viscous stress tensor}, Ann. of Math. (2), {\bf188} (2018), no. 2, 577--684

\bibitem{BGL} D. Bresch, M. Gisclon and I. Lacroix-Violet, {\em On Navier-Stokes-Korteweg and Euler-Korteweg Systems: Application to Quantum Fluids Models}. Arch. Rat. Mech. and Anal. {\bf 233} (2019), 975--1025.

\bibitem{BVY} D. Bresch, A. Vasseur and C. Yu, {\em Global Existence of Entropy-Weak Solutions to the Compressible Navier-Stokes Equations with Non-Linear Density Dependent Viscosities},  J. Eur. Math. Soc., {\bf 24} (2022), 1791-1837.

\bibitem{BH1} C. Burtea, and B. Haspot, {\em New effective pressure and existence of global strong solution for compressible Navier-Stokes equations with general viscosity coefficient in one dimension}. Nonlinearity {\bf 33}, 2077--2105 (2020)

\bibitem{BH2} C. Burtea and B. Haspot, {\em Existence of global strong solution for the Navier-Stokes-Korteweg system in one dimension for strongly degenerate viscosity coefficients}, Pure Appl. Anal., {\bf 4}, (2022), no. 3, 449-485.

\bibitem{CD} M. Caggio and D. Donatelli, {\em High Mach number limit for Korteweg fluids with density dependent viscosity}. Journal of Differential Equations, {\bf 277} (2021), 1-37. 

\bibitem{CD1} M. Caggio and D. Donatelli, {\em Relative entropy inequality for capillary fluids with density dependent viscosity and applications}. Math. Ann. (2024). 


\bibitem{CCH} R. Carles, K. Carrapatoso and M. Hillairet, {\em Rigidity results in generalized isothermal fluids}. 
Annales Henri Lebesgue {\bf 1}, 47--85, 2018.

\bibitem{CCH1} R. Carles, K. Carrapatoso and M. Hillairet, {\em Global weak solutions for quantum isothermal fluids}
Annales de l'Institut Fourier {\bf 72}(6), 2241--2298, 2022.

\bibitem{CH} F. Charve and B. Haspot, Existence of global strong solution and vanishing capillarity- viscosity limit in one dimension for the Korteweg system, SIAM J. Math. Anal. 45, (2013), no.2, 469-494.

\bibitem{CP2010} G.-Q. Chen and M. Perepelitsa, {\em Vanishing viscosity limit of the Navier-Stokes equations to the
Euler equations for compressible fluid flow}, Comm. Pure Appl. Math. {\bf63} (11) (2010), 1469--1504.

\bibitem{CDS} P. Constantin, Th. D. Drivas, and R. Shvydkoy. Entropy hierarchies for equations of compressible fluids and self-organized dynamics. SIAM J. Math. Anal., 52(3):3073–3092, 2020.

\bibitem{CDTP} Constantin, P., Drivas, T. D., Nguyen, H. Q., Pasqualotto, F.: Compressible fluids and active potentials. Ann. Inst. H. Poincar\'e Anal. Non Lin\'eaire 37, 145–180 (2020)

\bibitem{DPL} R.~J.~DiPerna and P.~L.~Lions,
{\em Ordinary differential equations, transport theory and Sobolev spaces.}
Invent. Math., {\bf 98} (1989), 511--547.

\bibitem{DFM} D. Donatelli, E. Feireisl and P. Marcati,{\em Well/ill posedness for the Euler-Korteweg-Poisson system and related problems}, Comm. Part. Differ. Equat., \textbf{40} (2015), 1314--1335.

\bibitem{DM} D. Donatelli and P. Marcati, {\em Quasineutral limit, dispersion and oscillations for Korteweg type fluids}, SIAM J. Math. Anal., \textbf{47} (2015), 2265--2282.

\bibitem{DM1} D. Donatelli and P. Marcati, {\em Low {M}ach number limit for the quantum hydrodynamics system},
Res. Math. Sci., \textbf{3} (2016), 2522--0144.

\bibitem{DS} J.E. Dunn and J. Serrin, \emph{On the thermomechanics of interstitial working}, Arch. Ration. Mech. Anal., {\bf 88} (1985), no. 2, 95--133.

\bibitem{EG} L.C. Evans and R.F. Gariepy, {\em Measure theory and fine properties of functions}, Studies in Advanced Mathematics, CRC Press, Boca Raton, FL, 1992.

\bibitem{F} E. Feireisl, {\em On compactness of solutions to the compressible isentropic Navier-Stokes equations when the density is not square integrable}, Comment. Math. Univ. Carolin., {\bf42} (2001), 83--98.

\bibitem{GLF2012} P. Germain and P. LeFloch, {\em Finite energy method for compressible fluids: The Navier-Stokes- Korteweg model}, Communications on Pure and Applied Mathematics, 69(1) (2016), 3--61.

\bibitem{GLV} M. Gisclon and I. Lacroix-Violet, {\em About the barotropic compressible quantum Navier-Stokes}
Nonlinear Anal., {\bf128} (2015),106--121.

\bibitem{GLT} J. Giesselmann, C. Lattanzio and A.-E. Tzavaras. {\em Relative energy for the Korteweg theory and related Hamiltonian flows in gas dynamics}, Arch. Ration. Mech. Anal., \textbf{223} (2017), 1427--1484.

\bibitem{GJX} Z. Guo, Q Jiu and Z. Xin, \emph{Spherically symmetric isentropic compressible flows with density-dependent viscosity coefficients}, SIAM J. Math. Anal. {\bf 39}, no. 5 (2008), 1402--1427.

\bibitem{JX} Q. Jiu and Z. Xin, {\em The Cauchy problem for 1D compressible flows with density-dependent viscosity coefficients}. Kinet. Relat. Models 1 (2008), no. 2, 313–330.

\bibitem{J} A. J\"ungel, {\em Global weak solutions to compressible Navier-Stokes equations for quantum fluids}, SIAM J. Math. Anal., {\bf42} (2010), 1025--1045.

\bibitem{J1} A. J\"ungel, \emph{Effective velocity in compressible Navier-Stokes equations with third-order derivatives}, Nonlin. Anal. Theory, Meth. Appl. {\bf 74}, no. 8 (2011), 2813--2818.

\bibitem{Kor} D.J. Korteweg, \emph{Sur la forme que prennent les \'equations du mouvement des fluides si l'on tient en compte des forces capillaires caus\'ees par des variations de densit\'e}, Archives N\'eerl. Sci. Exactes Nat. Ser. II {\bf 60} (1901), 1--24.

\bibitem{HS} R. Hagan and M. Slemrod, \emph{The viscosity-capillarity criterion for shocks and phase transitions}, Arch. Rat. Mech. Anal. {\bf 83}, no. 4 (1983), 333--361.


\bibitem{HL}  H. Hattori and D. Li, {\em Solutions for two dimensional system for materials of Korteweg type}, SIAM J. Math. Anal., {\bf25} (1994), 85--98.

\bibitem{HL1}H. Hattori and D. Li, {\em Global solutions of a high dimensional system for Korteweg materials}, J. Math. Anal. Appl., {\bf198} (1996), 84--97.

\bibitem{HM} M. Heida and J. M\'alek, \emph{On compressible Korteweg fluid-like materials}, Int. J. Eng. Sci. {\bf 48}, no. 11 (2010), 1313--1324.

\bibitem{H} D. Hoff, {\em Global solutions of the equations of one-dimensional, compressible flow with large data and forces, and with differing end states}. Zeitschrift f\"ur angewandte Mathematik und Physik ZAMP, {\bf 49}(5) (1998), 774--785.

\bibitem{KV} M.J. Kang and A. Vasseur: Global smooth solutions for 1D barotropic Navier-Stokes equations with a large class of degenerate viscosities. J. Nonlinear Sci. 30, 1703–1721 (2020)

\bibitem{BZZ} B. L\"u, R. Zhang and  X. Zhong, \emph{Global existence of weak solutions to the compressible quantum Navier-Stokes equations with degenerate viscosity}, J. Math. Phys. {\bf 60} (2019), 121502.

\bibitem{LV} I. Lacroix-Violet and A. Vasseur, {\em Global weak solutions to the compressible quantum Navier-Stokes equation and its semi-classical limit}, J. Math. Pures Appl., {\bf114} (2017), 191--210.

\bibitem{LLX} H.-L. Li, J. Li and Z. Xin, \emph{Vanishing of vacuum states and blow-up phenomena of the compressible Navier–Stokes equations}, Commun, Math. Phys. {\bf 281} (2008) 401--444.

\bibitem{LX} J. Li and Z. Xin, {\em Global Existence of Weak Solutions to the Barotropic Compressible Navier-Stokes Flows with Degenerate Viscosities}, Preprint: arXiv:1504.06826. 

\bibitem{LPZ} Y. Li, R. Pan and Z. Xin, \emph{On classical solutions to 2D shallow water equations with degenerate viscosities}, J. Math. Fl. Mech. {\bf 19}, no. 1 (2017), 151--190.

\bibitem{LXY}  T.-P. Liu, Z. Xin and T. Yang, \emph{Vacuum states for compressible flow}, Disc. Cont. Dyn. Syst. {\bf 4}, no. 1 (1998), 1--32.

\bibitem{LT} X. Liu and E. Titi, \emph{Global existence of weak solutions to the compressible primitive equations of atmospheric dynamics with degenerate viscosities}, SIAM J. Math. Anal. {\bf 51}, no. 3 (2019), 1913--1964.

\bibitem{L} P.L. Lions, {\em Mathematical Topics in Fluid Mechanics. Vol. 2.}, Claredon Press, Oxford Science Publications, 1996.

\bibitem{LPT1994} P.-L. Lions, B. Perthame and E. Tadmor, {\em Kinetic formulation of the isentropic gas dynamics and
$p$-systems}, Comm. Math. Phys. {\bf 163}, no. 2 (1994), 415--431.

\bibitem{Marche} F. Marche, \emph{Derivation of a new two-dimensional viscous shallow water model with varying topography, bottom friction and capillary effects}, Eur. J. Mech. B Fluids {\bf 26} (2007), 49--63.

\bibitem{MSi} P. Markowich and J. Sierra, {\em Non-uniqueness of weak solutions of the quantum-hydrodynamic system}, Kinet. Relat. Models {\bf 12}, (2019), no.2, 347--356.

\bibitem{MV} A. Mellet and A. Vasseur, {\em On the barotropic compressible Navier-Stokes equations}, Comm. Part. Differ. Equat., {\bf32} (2007), 431--452.

\bibitem{MV2} A. Mellet and A. Vasseur, {\em Existence and uniqueness of global strong solutions for one- dimensional compressible Navier-Stokes equations}. SIAM J. Math. Anal. {\bf 39}, 1344--1365

\bibitem{Sleu} V.V. Shelukhin, \emph{On the structure of generalized solutions of the one-dimensional equations of a polytropic viscous gas}, J. Appl. Math. Mech., 48(1984), 665–672; translated from Prikl. Mat. Mekh. 48(6), 912–920 (1984)

\bibitem{VY1} A. Vasseur and C. Yu, {\em Existence of global weak solutions for 3D degenerate compressible Navier-Stokes equations}, Invent. Math., {\bf206} (2015), 935--974.

\bibitem{VY2}A. Vasseur and C. Yu, {\em Global weak solutions to compressible quantum Navier-Stokes equations with damping}, SIAM J. Math. Anal., {\bf48} (2016), 1489--1511.

\bibitem{XZ1} Z. Xin and S. Zhu, \emph{Global well-posedness of regular solutions to the three-dimensional isentropic compressible navier-stokes equations with degenerate viscosities and vacuum}, Adv. Math. {\bf 393} (2021), 108072.

\bibitem{XZ2} Z. Xin and S. Zhu, \emph{Well-posedness of three-dimensional isentropic compressible Navier-Stokes equations with degenerate viscosities and far field vacuum}, J. Math. Pures Appl. {\bf 152} (2021), 94--144.

\end{thebibliography}
\end{document}